\documentclass[10pt]{article}

\usepackage[letterpaper, left=1in, top=1in, right=1in, bottom=1in, verbose, ignoremp]{geometry}

\usepackage{url}\RequirePackage[colorlinks,citecolor=blue, linkcolor=blue,urlcolor = blue]{hyperref}
\usepackage{latexsym,amssymb,amsmath,amsfonts,graphicx,color,fancyvrb,amsthm,enumerate,subfigure,mathrsfs}
\usepackage[authoryear,round]{natbib}
\usepackage[dvipsnames]{xcolor}
\usepackage{xy}\xyoption{all} \xyoption{poly} \xyoption{knot}
\usepackage{float}
\usepackage{bm}
\usepackage{bbm}
\usepackage{multicol,multirow}
\usepackage{booktabs}
\usepackage{array}
\usepackage{relsize}
\usepackage{chngcntr}
\usepackage{etoolbox}
\usepackage{caption}
\usepackage{tikz}
\usetikzlibrary{patterns}
\usepackage{amsopn}
\usepackage{calc}
\usepackage{algorithm}
\usepackage[noend]{algpseudocode}

\usepackage{hyperref}

\newtheorem{theorem}{Theorem}[]
\newtheorem*{theorem*}{Theorem}
\newtheorem{corollary}[theorem]{Corollary}
\newtheorem{lemma}[theorem]{Lemma}
\newtheorem{proposition}[theorem]{Proposition}

\newtheorem*{claim*}{Claim}
\theoremstyle{definition}
\newtheorem{definition}[theorem]{Definition}
\newtheorem*{definition*}{Definition}
\theoremstyle{AppDefinition}

\theoremstyle{AppClaim}

\theoremstyle{remark}
\newtheorem{remark}[theorem]{Remark}

\newtheorem{example}[theorem]{Example}
\newtheorem*{example*}{Example}

\def\beginmat{ \left( \begin{array} }
\def\endmat{ \end{array} \right) }

\def\log{{\rm log}}

\makeatletter
\newcommand*{\op}{%
  \DOTSB
  \mathop{\vphantom{\bigoplus}\mathpalette\matt@op\relax}%
  \slimits@
}
\newcommand\matt@op[2]{%
  \vcenter{\m@th\hbox{\resizebox{\widthof{$#1\bigoplus$}}{!}{$\boxplus$}}}%
}
\makeatother

\newcommand{\one}{{\bf{1}}}

\def\R{{\mathbb R}}

\newcommand{\Trop}{\mathrm{Trop}}
\newcommand{\Sym}{\mathrm{Sym}}

\makeatletter
\def\@biblabel#1{}
\makeatother

\makeatletter
\patchcmd{\NAT@citex}
  {\@citea\NAT@hyper@{%
     \NAT@nmfmt{\NAT@nm}%
     \hyper@natlinkbreak{\NAT@aysep\NAT@spacechar}{\@citeb\@extra@b@citeb}%
     \NAT@date}}
  {\@citea\NAT@nmfmt{\NAT@nm}%
   \NAT@aysep\NAT@spacechar\NAT@hyper@{\NAT@date}}{}{}

\patchcmd{\NAT@citex}
  {\@citea\NAT@hyper@{%
     \NAT@nmfmt{\NAT@nm}%
     \hyper@natlinkbreak{\NAT@spacechar\NAT@@open\if*#1*\else#1\NAT@spacechar\fi}%
       {\@citeb\@extra@b@citeb}%
     \NAT@date}}
  {\@citea\NAT@nmfmt{\NAT@nm}%
   \NAT@spacechar\NAT@@open\if*#1*\else#1\NAT@spacechar\fi\NAT@hyper@{\NAT@date}}
  {}{}

\makeatother

\begin{document}
\def\spacingset#1{\renewcommand{\baselinestretch}%
{#1}\small\normalsize} \spacingset{1}
%\spacingset{1.45} % DON'T change the spacing!
\begin{flushleft}
{\Large{\textbf{Tropical Geometric Variation of Phylogenetic Tree Shapes}}}
\newline
\\
Bo Lin$^{1}$, Anthea Monod$^{2,\dagger}$, and Ruriko Yoshida$^{3}$
\\
\bigskip
\bf{1} School of Mathematics, Georgia Institute of Technology, Atlanta, GA, USA 
\\
\bf{2} Department of Mathematics, Imperial College London, UK
\\
\bf{3} Department of Operations Research, Naval Postgraduate School, Monterey, CA, USA
\\
\bigskip
$\dagger$ Corresponding e-mail: a.monod@imperial.ac.uk\\
\end{flushleft}

%%%%%%%%%%%%%%%%%%%%%%%%%%%%%%%%%%%%%%%%%%%%%%%%%%%

\section*{Abstract}

We study the behavior of phylogenetic tree shapes in the tropical geometric interpretation of tree space.  Tree shapes are formally referred to as tree topologies; a tree topology can also be thought of as a tree combinatorial type, which is given by the tree's branching configuration and leaf labeling.  We use the tropical line segment as a framework to define notions of variance as well as invariance of tree topologies: we provide a combinatorial search theorem that describes all tree topologies occurring along a tropical line segment, as well as a setting under which tree topologies do not change along a tropical line segment.  Our study is motivated by comparison to the moduli space endowed with a geodesic metric proposed by Billera, Holmes, and Vogtmann (referred to as BHV space); we consider the tropical geometric setting as an alternative framework to BHV space for sets of phylogenetic trees.  We give an algorithm to compute tropical line segments which is lower in computational complexity than the fastest method currently available for BHV geodesics and show that its trajectory behaves more subtly: while the BHV geodesic traverses the origin for vastly different tree topologies, the tropical line segment bypasses it.

\paragraph{Keywords:} BHV geodesic metric; Clades; Tree topology; Tropical geometry; Tropical line segment.

%%%%%%%%%%%%%%%%%%%%%%%%%%%%%%%%%%%%%%%%%%%%%%%%%%%

\section{Introduction}
\label{sec:intro}

Phylogenetic trees are discrete mathematical objects that capture the evolutionary behavior of biological processes; they have been extensively used and studied both as data structures as well as symbolic objects.  Phylogenetic trees are able to encode vastly different evolutionary patterns through various branching configurations, which define the tree's shape or combinatorial type---more formally, the tree's {\em topology}.  In this paper, we focus on phylogenetic trees as discrete geometric objects and study the behavior of their topology using tropical geometry.  Specifically, we propose the {\em tropical line segment} as a framework to study how tree topologies vary.  Our main contributions are a detailed study of the occurrence and behavior of tree topologies and a method to compute tropical line segments in tropical geometric phylogenetic tree space.  Our study is conducted relative to the current standard for the space of phylogenetic trees, proposed by Billera, Holmes, and Vogtmann \citep{BILLERA2001733}, and referred to as {\em BHV space} (after the authors' initials).  BHV space is a moduli space of phylogenetic trees whose geometry is characterized by unique geodesics and whose structure is defined by tree topologies.  

We prove a combinatorial theorem that describes all tree topologies that exist along the tropical line segment between any two given trees, as well as a framework for the notion of invariance of tree topologies on a tropical line segment.  We also give an algorithm that computes tropical line segments and compare tropical line segments to BHV geodesics.  We find that the tropical setting results in a more subtle behavior of trajectories between vastly different tree topologies.

The remainder of our paper is organized as follows.  In Section \ref{sec:trop}, we give formalities on phylogenetic trees and tree spaces; we also provide the basics of tropical geometry and discuss the coincidence of tropical geometry and tree space.  A review on the literature on comparisons between trees is also given, as well as a detailed overview of BHV space and its geometry.  In Section \ref{sec:structure}, we define the tropical line segment and characterize the geometry of phylogenetic tree space in the tropical geometric setting.  In Section \ref{sec:main}, we provide our main results on tree topologies on tropical line segments, which are a combinatorial theorem that provides all tree topologies occurring along a tropical line segment and a notion of invariance of tree topologies on tropical line segments.  We also present an algorithm to compute tropical line segments and compare their behavior relative to BHV geodesics in terms of trajectories across tree space and between different tree topologies.  We show that the tropical line segment does not cross the origin, which differs from the trajectory of BHV geodesics between trees with vastly different tree topologies.  We close with a discussion in Section \ref{sec:end}. 

%%%%%%%%%%%%%%%%%%%%%%%%%%%%%%%%%%%%%%%%%%%%%%%%%%%

\section{Trees, Tree Spaces, and Tropical Geometry}
\label{sec:trop}

In this section, we present the basics of tropical geometry, provide background on phylogenetic trees and tree spaces, and discuss the coincidence between tropical geometry and tree spaces.  We overview how trees are compared within tree spaces, focusing on the setting of BHV space and its geometry.  Finally, we define the tropical line segment and discuss some of its geometric properties.

\paragraph{Notation.}  In this paper, $N \in \mathbb{N}$.  We use the notation $[N] := \{1, 2, \ldots, N\}$ and $n := \binom{N}{2}$.  Also, we write interchangeably $w(i,j) = w_{ij} = w_{\{i,j \}}$.

\subsection{Tropical Algebra and Tropical Geometry}

Tropical algebra is based on the following semiring with linearizing operations given as follows.

\begin{definition}
The {\em tropical (max-plus) semiring} is $(\R \cup \{ -\infty \}, \boxplus, \odot)$ with addition and multiplication defined by
\begin{align*}
a \boxplus b & := \max(a,b)\\
a \odot b & := a + b.
\end{align*}
\end{definition}

Existing literature on tropical geometry specifies a min-plus semiring, where the addition of two elements is given by their minimum and denoted by $\oplus$, rather than their maximum: $(\R \cup \{\infty \}, \oplus, \odot)$ \citep{speyer2009tropical}.  The min-plus and max-plus semirings are isomorphic.  In the context of phylogenetic trees, the max-plus semiring is a more appropriate convention to adopt and consistent with existing literature on tropical geometric methods in tree spaces \citep[e.g.,][]{Yoshida2018,monod2018tropical,tang2020tropical}.

The tropical operations are commutative and associative, and multiplication distributes over addition.  Tropical subtraction is not defined, which gives a semiring rather than a ring.  {\em Tropicalization} refers to replacing classical arithmetic operations with the tropical versions.  Tropical geometry is the study of the geometry of nonlinear loci of polynomial systems defined in the tropical semiring.

\subsection{Defining Phylogenetic Trees}

Phylogenetic trees are the fundamental mathematical model for biological evolution.  They are constructed from molecular sequence data of a finite number of species, and graphically represent the evolutionary phylogeny of the species.

\begin{definition}
A {\em phylogenetic tree} $T = (V,E)$ is an acyclic connected graph with at most one vertex of degree 2.  $V$ is a set of vertices that are labeled terminal nodes with degree 1 called {\em leaves}; non-leaf vertices have degree greater than 2.  $E$ is a set of nonnegative length {\em edges} or {\em branches} that represent evolutionary time.  Edges that connect to leaves are called {\em external} or {\em pendant} edges; otherwise, they are known as {\em internal} edges.

When there is a common ancestor from which all leaves evolve, the tree is called {\em rooted} and the {\em root} is a unique node of degree 2.  In rooted trees, the evolution progresses from the common ancestor (root) by a series of bifurcations (edges, $E$) and ends in the terminal nodes (leaves, $V$).  When there is no common ancestor among the leaves, the tree is {\em unrooted}.
\end{definition} 

\begin{remark}
\label{rem:root}
In the case of rooted trees, one may also imagine an edge extending from the root to a leaf labeled 0.  In this case, the interior vertex connecting to the root (leaf label 0) will have degree 3, however, by convention, rooted trees are not depicted this way, which is why the root in such trees appears to be a node of degree 2. 
\end{remark}

Methods for reconstructing phylogenetic trees from molecular sequence data are generally either distance-based or statistical.  Distance-based methods entail specifying distance matrices between the sequences via a genetic distance---such as Hamming distance---and grouping sequences that are closely related under the same node, with branch lengths representing the observed distances between sequences.  For a survey on distance-based tree reconstruction methods, see for example \cite{peng2007distance}.

Statistical reconstruction methods entail specifying some classical statistical criterion, such as likelihood or parsimony, and then optimizing \citep[e.g.,][]{max_pars,Felsenstein1981}.  The criteria are defined on the principle that in DNA evolution, nucleotides are substituted following a continuous-time Markov chain (or, more generally, an independent time-reversible model for finite sites \citep{tavare1986some}).  The motivation for statistical methods for tree reconstruction arises from the uncertainty of the ``true" phylogenetic tree, since different choices of molecular sequences (which may be due to choice of gene or coding region) leads to different gene trees \citep[e.g.,][]{HOLMES200317}.  Additionally, there is an extremely high number of possible tree topologies \citep{schroder1870vier}: the number of tree topologies for a rooted, binary tree (i.e., a bifurcating tree with exactly two descendants stemming from each interior node) with $N$ leaves is
\begin{equation}
\label{eq:num_top}
(2N-3)!! = (2N-3) \times (2N-5) \times \cdots \times 3.
\end{equation}  

Previous work has shown that solutions to statistical optimization problems are tractable under certain conditions and assumptions.  For example, under uniform distributivity, the optimization of parsimony-based objective functions is known to be an NP-complete Steiner tree problem \citep{FOULDS198243}.  Various restrictions of the Steiner tree problem (e.g., the minimum spanning tree problem) can be solved in polynomial time \citep{juhl2014geosteiner}.  In biological applications where the data and specific problem of study allow certain distributional assumptions (e.g., identifiability of mixture distributions), statistical methods can be easy and computationally efficient to implement \citep[e.g.,][]{allman2008phylogenetic,long2015identifiability,rhodes2012identifiability}.  The focus of this paper, however, is on the number of tree topologies (\ref{eq:num_top}) and their occurrences within tree spaces.

\begin{definition}
For a phylogenetic tree $T$ with $N$ leaves labeled by $[N]$, and with branch length $b_e \in \R_{\geq 0}$ associated to each edge $e$ in $T$, its {\em tree metric} is the map
\begin{align*}
w: [N] \times [N] & \rightarrow \R_{\geq 0}\\
w(i,j) & \mapsto \bigodot_{e \in P_{\{i,j \}}} b_e,
\end{align*}
where $P_{\{ i,j \}}$ is the unique path between leaves $i$ and $j$.
\end{definition}

Tree metrics are metric representations of phylogenetic trees in terms of pairwise distances between leaves.  Tree metrics may also be represented as {\em cophenetic vectors} \citep{Cardona2013},
\begin{equation}
\label{eq:coph}
(w(1,2),\, w(1,3),\, \ldots,\, w(N-1, N)) \in \R^n_{\geq 0},
\end{equation}
where the entries are sorted lexicographically.

\begin{definition}[Four-point condition, \cite{BUNEMAN197448}]
A tree metric satisfies the {\em four-point condition} and hence defines a tree if and only if the maximum among the {\em Pl\"ucker relations},
\begin{equation}
\label{eq:plucker1}
\begin{split}
w(i,j) \odot w(k, \ell),\\
w(i,k) \odot w(j, \ell),\\
w(i, \ell) \odot w(j,k),
\end{split}
\end{equation}
is attained at least twice for $1 \leq i < j < k < \ell \leq N$, or, equivalently, if
\begin{equation}
\label{eq:plucker2}
w(i,j) \odot w(k, \ell) \leq w(i,k) \odot w(j, \ell) \boxplus w(i, \ell) \odot w(j,k)
\end{equation}
for all distinct $i, j, k, \ell \in [N]$.
\end{definition}

The above technical condition characterizes phylogenetic trees; in particular, the {\em space of phylogenetic trees with $N$ leaves}, $\mathcal{T}_N$, is the collection of all $n$-tuples $\{w(i,j)\}_{1 \leq i < j \leq N}$ that satisfy the four-point condition (\ref{eq:plucker2}), or equivalently, where the maximum among (\ref{eq:plucker1}) is achieved at least twice.  For examples and counterexamples of the four-point condition with illustrations, see \cite{monod2018tropical}.  

The coincidence between the space of phylogenetic trees and tropical geometry arises through the four-point condition as follows.  \cite{Speyer2004} identify a homeomorphism between $\mathcal{T}_N$ and a tropical version of the Grassmannian of 2-planes in $N$ dimensions: the Grassmannian may be mapped to a projective variety via the Pl\"{u}cker embedding, which, when interpreted tropically, gives the homeomorphism by \cite{Speyer2004}, by recovering the four-point condition defining phylogenetic trees.  This endows the space of phylogenetic trees with a tropical geometric structure; in particular, the space of phylogenetic trees is a tropical variety.

The four-point condition may be strengthened to define an important subclass of trees as follows.

\begin{definition}[Three-point condition, \cite{JARDINE1967173}]
A tree metric satisfies the {\em three-point condition} and hence defines a {\em tree ultrametric} if and only if the maximum among
$$
w(i,j), \quad
w(i,k), \quad
w(j,k)
$$
is attained at least twice for $1 \leq i < j < k \leq N$.
\end{definition}

The {\em space of tree ultrametrics with $N$ leaves}, $\mathcal{U}_N$, is the collection of all $n$-tuples $\{ w(i,j) \}_{1 \leq i < j \leq N}$ satisfying the three-point condition.  A rooted phylogenetic tree $T$ is {\em equidistant} if the distance from every leaf to its root is constant.

\begin{proposition}[Proposition 12, \cite{monod2018tropical}]
\label{prop:ultra}
A tree metric $w$ for a phylogenetic tree $T$ is a tree ultrametric if and only if $T$ is equidistant.
\end{proposition}

\subsection{The Tropical Projective Torus}

For equidistant trees, since the distance between the root and every leaf is constant, this distance (or the tree's height) may always be normalized to 1.  This idea may be generalized to unrooted trees to normalize the evolutionary time between trees by considering an equivalence relation for tree metrics represented by cophenetic vectors $x,y \in \R^n$:
$$
x \sim y \Leftrightarrow x_1 - y_1 = x_2 - y_2 = \cdots = x_n - y_n,
$$
meaning that $x \sim y$ if and only if all coordinates of their difference $x - y$ are equal.  This equivalence relation generates a quotient space known as the {\em tropical projective torus}, denoted by $\R^n/\R\one$, which is the ambient space of tree space, $\mathcal{U}_N \subset \mathcal{T}_N \subset \R^n/\R\one$.  The tropical projective torus $\R^n/\R\one$ may be embedded into $\R^{n-1}$ by considering representatives of the equivalence classes with first coordinate equal to zero, $(x_2 - x_1, x_3 - x_1, \ldots, x_n - x_1) \in \R^{n-1}$.  See \cite{maclagan2015introduction,monod2018tropical,lee2019tropical} for more detail.

\subsection{Comparing Trees via Metrics}
\label{subsec:metrics}

Metrics are a natural tool for comparing trees and providing a quantitative measure of similarity between trees.  Various metrics have been proposed over the past several decades during which quantitative and computational tree studies have been an active research interest.  

Some of the most popular metrics between trees are those that maintain characteristics of Euclidean distance in the inherently non-Euclidean setting of tree space, such as an inner product structure.  Some examples of well-known inner product distances between trees are the path difference, quartet distance, and Robinson--Foulds distance \citep{ROBINSON1981131}.  However, these are known to suffer from structural errors, since many pairs of trees measure the same distance apart under these metrics; as well as interpretive errors, since the existence of large distances between trees does not necessarily mean that there are large differences between shared ancestry of leaves \citep{Steel1993}.  Other metrics use the cophenetic vector representation of tree metrics, treating them simply as points in Euclidean space, and then using the $\ell^\infty$ distance in $\R^n$ \citep{Cardona2013}.  

\begin{remark}
\label{rem:linf}
Notice that the image of the linear mapping from $\mathbb{R}^N$ to $\mathbb{R}^n$ given by
\begin{equation}
\label{eq:linmap}
(x_1, \ldots, x_N) \mapsto (x_i - x_j),
\end{equation}
for all pairs $i < j$, is isomorphic to the tropical projective torus.  Under such a mapping, it is then possible to work in $\mathbb{R}^n$ equipped with the $\ell^\infty$ metric as in \cite{ardila2005subdominant,bernstein2017infinity,bernstein2020infinity}.
\end{remark}

Other popular and well-known metrics include the nearest neighbor interchange metric \citep{waterman1976some}, subtree transfer distance \citep{Allen2001}, and variational distance \citep{steel2006variational}.  There also exist metrics that connect the subfield of mathematical phylogenetics to various other subfields of pure mathematics, such as the algebraic metric \citep{ALBERICH20091320}, which is based on a group structure; \cite{munch2019} also show that the $\ell^\infty$ metric applied to trees \citep{Cardona2013} is in fact an interleaving distance from applied and computational topology.

\paragraph{BHV Space and the Geodesic Metric.}
\label{description:BHV}

We focus and provide more detail on the geometric interpretation of tree space proposed by \cite{BILLERA2001733} for two main reasons.  First, the tropical geometric approach is the most comparable to the BHV setting since it is also geometric in nature and also defines a moduli space.  Second, previous work which this paper builds upon is in direct comparison to BHV space; we follow suit for consistency. 

In BHV space, $\mathcal{T}_N^{\mathrm{BHV}}$, phylogenetic trees are represented as Euclidean vectors; the entries in the vectors are given by internal edge lengths of the trees, external edge lengths are disregarded.  For a phylogenetic tree with $N$ leaves, there are at most $2N-2$ edges: $N$ terminal edges connecting to leaves, and at most $N-2$ internal edges.  Trees in BHV space are therefore represented as vectors in $\R^{N-2}$, for a given tree topology.  The space of phylogenetic trees is then a collection of $(2N-3)!!$-many Euclidean orthants; each orthant may be regarded as the polyhedral cone of $\R^{N-2}$ with all nonnegative coordinates, which correspond to the internal edge lengths in a tree.  

On the orthant boundaries, there is at least one zero coordinate that occurs; orthant boundaries represent trees with collapsed internal edges.  The orthants are joined along the orthant boundaries.  In general, tree topologies of the orthants determine the adjacency: boundary trees from two different orthants (i.e., trees with two different topologies) may characterize the same polytomic topology (or split) so these two orthants are grafted together along this boundary when the trees from each topology with collapsed internal edges coincide.  Two adjacent orthants, therefore, represent similar, yet distinct, tree topologies.  Orthants are grafted at right-angles, resulting in a stratified space.

The right-angle grafting has a direct implication on the curvature of BHV space; BHV space satisfies the $\mathrm{CAT}(0)$ property of flag complexes, and geodesics are thus unique.  The geodesic characterizes a metric on BHV space, where the length of the geodesic between any two trees is the distance.  Geodesics are computable on BHV space; the geodesic between any two trees represented by their internal edge lengths is first computed, external edges are then factored in afterwards to compute the overall distance.  For two trees in the same orthant, the shortest path between them is simply the straight line measured by the Euclidean distance between them.  For trees in different orthants, the difficulty arises in establishing the sequence of orthants to traverse to give the shortest distance between the trees.  For trees that are not in neighboring orthants and especially when the tree topologies are very different, the geodesic often passes through the origin (or star tree, where all internal edges are collapsed); paths that traverse the origin are referred to as {\em cone paths}.  For trees with four leaves, the sequence of orthants containing the shortest path between two trees can be systematically computed by a grid search, but for larger trees, such a search is intractable.  \cite{Owen:2011:FAC:1916480.1916603} give the fastest available algorithm to date to find the geodesic path between any two trees in BHV space, which runs in quartic time in the number of leaves $N$.

%\begin{remark}
%There also exist trees in BHV space with no specific boundary assigned to their topology.
%\end{remark}

\subsection{Palm Tree Space}

\cite{monod2018tropical} present an alternative geometric construction of phylogenetic tree space based on tropical geometry and study its analytic and topological properties with the aim of statistical inference and data analysis in mind.  Trees are represented by cophenetic vectors (\ref{eq:coph}); external edge lengths are thus included.  Endowed with the {\em tropical metric}, the resulting metric moduli space is referred to as {\em palm tree space} (tropical tree space).

The tropical metric is a generalized Hilbert projective metric function and arises in other tropical geometric settings \cite[e.g.,][]{AKIAN20113261,COHEN2004395}.  It has also been used in previous work studying phylogenetic trees \citep{doi:10.1137/16M1079841,lin2016tropical,Yoshida2018}.  The tropical metric is combinatorial in nature and is given by the difference between the maximum and minimum of the differences between the tree coordinates; it is a proper metric \citep{monod2018tropical}.  This metric also enjoys other interpretations following the relationship of isomorphism discussed in Remark \ref{rem:linf}; in particular, restricting to the image of the map (\ref{eq:linmap}), the tropical metric is precisely the $\ell^\infty$ metric.  

Under the tropical metric, properties that are desirable for statistical inference (such as hypothesis testing) and exact probabilistic studies (such as concentration inequalities and convergence studies) are satisfied and well-defined.  Geodesics, however, are not unique; there are infinitely many geodesics between any two points.  This indicates a more complex geometry than the $\mathrm{CAT}(0)$ structure of BHV space.  \cite{lee2019tropical} use optimal transport theory to define and study the Wasserstein distances on the tropical projective torus and thus give an algorithm to compute the set of all infinitely-many geodesics on the tropical projective torus.

%%%%%%%%%%%%%%%%%%%%%%%%%%%%%%%%%%%%%%%%%%%%%%%%%%%

\section{Structure and Geometry of Tropical Geometric Tree Space}
\label{sec:structure}

BHV space and palm tree space are both moduli spaces and are both inherently geometric constructions.  In this section, for comparative purposes and to present the framework of our main results, we study the structure and geometry of the tropical geometric interpretation of phylogenetic tree space.

\subsection{Polyhedral Structure}

In the same way that $\mathcal{T}_N^{\mathrm{BHV}}$ is constructed as the union of polyhedra, where each polyhedron corresponds to one distinct tree topology, the structure of $\mathcal{T}_N$ interpreted tropically is also given by such a union.

\begin{proposition}[Proposition 4.3.10, \cite{maclagan2015introduction}]
\label{prop:polyhedra}
The space $\mathcal{T}_{N}$ is the union of $(2N-5)!!$ polyhedra in $\mathbb{R}^n/\mathbb{R}\one$ with dimension $N-3$.
\end{proposition}

As in BHV space, in tropical geometric tree space, each polyhedron may be considered as a polyhedral cone of $\R^{N-3}$ with all nonnegative coordinates corresponding to the cophenetic vector representation of trees.  Here, the cone in consideration is the usual convex cone as a subset of a vector space on an ordered field, closed under linear combinations with positive coefficients (or equivalently, the set spanned by conical combinations of vectors).

\begin{example}

\begin{figure}[h]
	\centering
	\includegraphics[scale=0.4]{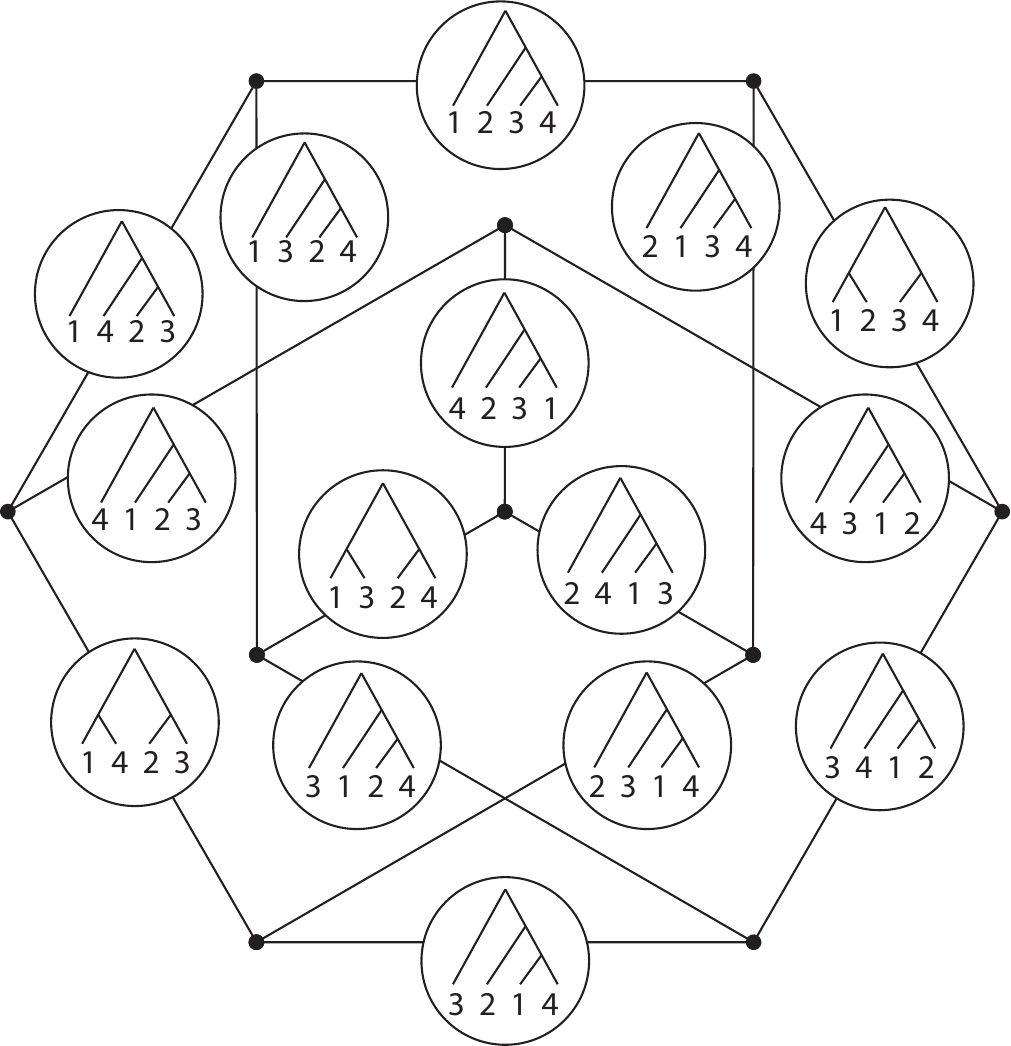}
	\caption{The Petersen graph depicting the $15$ cones in $\mathcal{T}_{5}$.}\label{fig:petersen}
\end{figure}

When $N=5$, $\mathcal{T}_5$ has $5!!=15$ cones, which correspond to the edges of the {\em Petersen graph}, depicted in Figure \ref{fig:petersen}.  Here, the root is considered as a leaf (see Remark \ref{rem:root} above), and the other leaf labels are numbered 1 through 4.  %In this example, Remark \ref{rmk:root} is relevant, since the $0$-node is disregarded in representing the rooted trees in terms of their leaf labels $1$ through $4$.
The Petersen graph here is illustrative and represents the configuration of the polyhedra according to tree topologies; the edges represent the cones corresponding to a tree topology, the nodes of the graph represent the commonality between the coinciding edges (i.e., between two tree topologies).

To interpret the correspondence to trees in this Petersen graph, consider the leftmost upper graph vertex on the outer hexagon and notice that there are three edges (cones) that meet at this vertex.  The figures of the trees associated with these cones (illustrated in the circles) share the property that the edge to leaf $1$ is the longest edge, and the remaining three leaves are permuted.  Now, consider the graph vertex at the very center of the graph, inside the hexagon: the property common to the trees associated with the three cones meeting at this vertex is that the pair of leaves labeled $1$ and $3$ are coupled, symmetric, and are joined at the same internal node of the tree that is not the root.  The remaining graph vertices may be interpreted in a similar manner, in the sense that they all share one of these two commonalities, under symmetry of and up to leaf labeling scheme.  Thus, the $10$ vertices of the Petersen graph here are given by the number of ways to choose a label for the longest branch length in a tree that connects directly to the root, among $4$ choices of leaf labels (i.e., $\binom{4}{2}$); and the number of ways to choose pairs of leaf labels that are coupled, symmetric, and correspond to the same internal node of the tree that is not directly linked to the root (i.e., $\binom{4}{1}$), so $\binom{4}{2} + \binom{4}{1} = 10$.  The intuition of the graph edges as cones lies in the property of closedness under scaling of branch lengths, for each tree type associated with each graph edge in the figure (illustrated in the circles). 

The Petersen graph also coincides in the context of BHV space when considering $N=4$ leaves, as the so-called {\em link of the origin}; see \cite{BILLERA2001733} for further details.  %More generally, the structures of palm tree space and BHV space coincide for arbitrary $N$.
\end{example}

\subsection{Geometry of the Space of Tree Ultrametrics}
\label{subsec:ultra}

We now focus on the case of rooted equidistant trees as in \cite{BILLERA2001733}.  We begin by providing definitions of tropical geometric line and set objects and outlining important properties of these objects in the space of tree ultrametrics.

\paragraph{Notation.}  For a positive integer $N$, let $p_{N}$ be the set of all pairs in $[N]$.  For convenience, we denote a tree ultrametric with $N$ leaves by $(w_{p})_{p\in p_{N}}$, where for $p=\{i,j\}$, $w_{p} = w(\min(i,j),\max(i,j))$.

\begin{definition}
\label{def:3pt}
Consider the subspace of $L_N \subseteq \mathbb{R}^n$
defined by the linear equations
\begin{equation}
\label{eq:trop_eq}
w_{ij} - w_{ik} + w_{jk}=0
\end{equation} 
for $1\leq i < j < k \leq N$ in tree metrics $w$.  
For the linear equations (\ref{eq:trop_eq}) cutting out $L_N$, their (max-plus) tropicalization is $w_{ij} \boxplus w_{ik} \boxplus w_{jk}$: recall that under the trivial valuation, all coefficients are disregarded when tropicalizing.  This tropicalization of $L_N$ is denoted by $\Trop(L_N) \subseteq \mathbb{R}^{n}/\mathbb{R}\one$ and is referred to as the {\em tropical linear space} with points $(w_{p})_{p\in p_{N}}$ where $\max\big(w_{ij},\, w_{ik},\, w_{jk} \big)$ is obtained at least twice for all triples $i,\,j,\,k \in [N]$.
\end{definition}

This is equivalent to the three-point condition for ultrametrics given in Definition \ref{def:3pt}.  An observation from tropical geometry gives a correspondence between the tropical linear space $\Trop(L_N)$ and the graphic matroid of a complete graph with $N$ vertices. 

We have the following geometric characterization of the space $\mathcal{U}_N$ and corresponding characterizations of tropical line segments between ultrametrics.

\begin{theorem}[\cite{ARDILA200638}]
\label{tropicalLine}
The image of $\mathcal{U}_N$ in the tropical projective torus $\mathbb{R}^{n}/\mathbb{R}\one$ coincides with $\Trop(L_N)$.  That is, $\Trop(L_N) = \mathcal{U}_N$.
\end{theorem}

\begin{definition}
For $x, y \in \R^n/\R\one$, the {\em tropical line segment} with endpoints $x$ and $y$ is the set
$$
\{ a \odot x \boxplus b \odot y \in \R^n/\R\one \mid a,b \in \R \}.
$$
Here, max-plus addition $\boxplus$ for two vectors is performed coordinate-wise.
\end{definition}

The tropical line segment between any two points in $\R^n/\R\one$ is unique and it is a geodesic \citep{monod2018tropical}.  

\begin{definition}
Let $S \subset \mathbb{R}^n$.  If $a \odot  x \boxplus b \odot y \in S$ for all $x,y \in S$ and all $a,b \in \mathbb{R}$, then $S$ is said to be {\em tropically convex}. 

The {\em tropical convex hull} or {\em tropical polytope} of a given subset $V \subset \mathbb{R}^n$ is the smallest tropically-convex subset containing $V \subset \mathbb{R}^n$; it is denoted by $\mathrm{tconv}(V)$.  The tropical convex hull of $V$ may be also written as the set of all tropical linear combinations:
$$
\mathrm{tconv}(V) = \{a_1 \odot v_1 \boxplus a_2 \odot v_2 \boxplus \cdots \boxplus a_n \odot v_n \mid v_1,\ldots,v_n \in V \mbox{ and } a_1,\ldots,a_n \in \mathbb{R} \}.
$$
\end{definition}

\begin{proposition}\label{prop:tropline}
For two tree ultrametrics $T_{1},T_{2}\in \mathcal{U}_{N}$, the tropical line segment generated by $T_{1}$ and $T_{2}$,
$$
a \odot T_1 \boxplus b \odot T_2~\forall~a, b \in \mathbb{R},
$$
is contained in $\mathcal{U}_N$.
%$$
%\{ T_1, T_2 \in \mathcal{U}_N \mbox{~and~} a, b \in \R \mid a\odot T_{1} \boxplus b\odot T_{2} \in \mathcal{U}_N \}.
%$$
In other words, $\mathcal{U}_N$ is tropically convex.
\end{proposition}

\begin{proof}
Since $T_1$ is a tree ultrametric, $a\odot T_{1}$ remains a tree ultrametric; this is also true for $b\odot T_{2}$.  Thus, we may assume $a=b=0$.

Suppose $T_{1}=(w_{\{i,j\}})_{1\le i<j\le N}$ and $T_{2}=(w'_{\{i,j\}})_{1\le i < j\le N}$, then $T_{1} \boxplus T_{2}=(w_{\{i,j\}} \boxplus w'_{\{i,j\}})_{1\le i < j\le N}$.  Let $z_{\{i,j\}}= w_{\{i,j\}} \boxplus w'_{\{i,j\}}$.  It suffices to show that for any $1\le i<j<k\le n$, we have that the maximum among $z_{\{i,j\}},\, z_{\{i,k\}},\, z_{\{j,k\}}$ is attained at least twice in order for $T_1 \boxplus T_2$ to be a tree ultrametric.  Let $M$ be this maximum, and set $M := z_{\{i,j\}}$.  Then either $w_{\{i,j\}}=M$ or $w'_{\{i,j\}}=M$.  If $w_{\{i,j\}}=M$, at least one of $w_{\{i,k\}},\, w_{\{j,k\}}$ is equal to $M$ since $T_{1}$ is a tree ultrametric, therefore at least one of $z_{\{i,k\}},\, z_{\{j,k\}}$ is also equal to $M$.  Similarly, if $w'_{\{i,j\}}=M$, at least one of $w'_{\{i,k\}},\, w'_{\{j,k\}}$ is equal to $M$, since $T_{2}$ is a tree ultrametric, therefore at least one of $z_{\{i,k\}},\, z_{\{j,k\}}$ is also equal to $M$.  Thus, in either case the maximum among $z_{\{i,j\}},\, z_{\{i,k\}},\, z_{\{j,k\}}$ is attained at least twice, hence $T_{1} \boxplus T_{2}$ is a tree ultrametric. 
\end{proof}

This result generalizes outside the context of trees; in general, tropical linear spaces in the tropical projective torus are tropically convex (see Proposition 5.2.8 of \cite{maclagan2015introduction}).

%%%%%%%%%%%%%%%%%%%%%%%%%%%%%%%%%%%%%%%%%%%%%%%%%%%

\section{Tree Topologies on Tropical Line Segments}
\label{sec:main}

In this section, we present our main results, which include a combinatorial study of the variation of tree topologies as well as a notion of invariance within the framework of the tropical line segment.  We also give an algorithm for computing tropical line segments.

Note that our study is conducted in the setting of rooted equidistant trees (ultrametrics).  The relevance for a study dedicated to ultrametrics specifically in this paper is twofold.  First, it allows for a parallel geometric comparison between BHV space and the tropical interpretation of phylogenetic tree space: As mentioned above in Section \ref{description:BHV}, the geometric significance of BHV space lies in its construction when only internal edges are considered.  Its structure is based on the union of orthants, where each orthant corresponds to a specific rooted tree topology.  In other words, the definition of BHV space inherently relies on tree topologies.  The tropical construction of tree space, while also polyhedral (see Proposition \ref{prop:polyhedra}), has a more complex algebraic structure, which we explore here by studying rooted equidistant tree topologies and their occurrence within the tropical construction of tree space.

The second motivation for studying ultrametric tree topologies lies in the context of applications: ultrametrics correspond to {\em coalescent processes}, which model important biological phenomena, such as cancer evolution \citep{Kingman1461}.  In phylogenomics, the {\em coalescent model} is often used to model gene trees given a species tree (see e.g., \cite{Knowles2009,Rosenberg2003,Tian2014} for further details).  The coalescent model takes two parameters, the population size and species depth (i.e., the number of generators from the most recent common ancestor (MRCA) of all individuals at present).  The species depth coincides with the height of each gene tree from the root (that is, the MRCA) to each leaf representing an individual in the present time.  The output of the coalescent model is a set of equidistant gene trees, since the number of generations from their MRCA to each individual in the present time are the same by model construction.  Understanding the structure of ultrametrics is an important step towards the modeling and analysis of coalescent biological processes.

%\paragraph{Notation.}  For a positive integer $N$, let $p_{N}$ be the set of all pairs in $[N]$.  For convenience, we denote a tree ultrametric with $N$ leaves by $(w_{p})_{p\in p_{N}}$, where for $p=\{i,j\}$, $w_{p} = w(\min(i,j),\max(i,j))$.

\subsection{Tree Variation on Tropical Line Segments}
\label{sec:tree_topology}

The tropical linear space coincides with the space of tree ultrametrics, and hence, that tropically-convex sets (and therefore tropical line segments) are also fully contained in the space of tree ultrametrics.  This endows the space of tree ultrametrics with a tropical structure, and now allows us to study the behavior of points (trees) along tropical line segments.  In particular, this allows us to characterize ultrametric tree topologies geometrically, thereby providing a description of the tropically-constructed tree space that is comparable to the geometry of BHV space for rooted trees and zero-length external edges (as originally described by \cite{BILLERA2001733}).

The strategy that we implement is largely combinatorial.  We first formalize the definition of a tree topology as a collection of subsets of leaves, and use these subsets to define notions of size, and in particular, largest and smallest subsets.  Given these upper and lower bounds, we then define an equivalence relation and a partial order that allow us to iteratively and combinatorially partition and compare leaf subsets and tree topologies.  This gives us a framework to study shapes of trees: specifically, Theorem \ref{thm:compatible} is a combinatorial theorem that describes the possible tree topologies that exist along a tropical line segment. 

\begin{definition}\label{def:treetopology}
A {\em tree topology} $F$ on $[N]$ is a collection of subsets called {\em clades} $S \subseteq [N]$, where $2\le |S|\le N-1$ and for any two distinct clades $S_{1},S_{2}\in F$, exactly one of the following {\em nested set conditions} holds:
\begin{equation}\label{eq:nested}
\begin{split}
S_{1} & \subsetneq S_{2},\\
S_{2} & \subsetneq S_{1},\\
S_{1} \cap S_{2} & =\emptyset.
\end{split}
\end{equation}
$F$ is said to be {\em full dimensional} if $|F|=N-2$.
\end{definition}

Clades always belong to $[N]$ rather than $[N] \cup \{0\}$, since we may always choose the clade excluding the root (i.e., the leaf with label $0$).  In this manner, they allow for an alternative representation of trees over tree metric vectors $w$ or matrices $W$.

\begin{example}

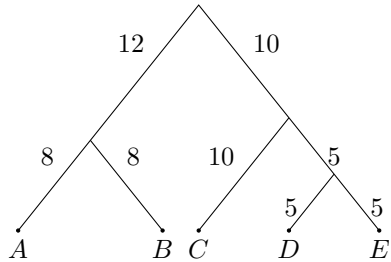
\begin{figure}[h]
			\centering
			\begin{tikzpicture}[scale=0.6]
				\draw (0,0) -- (4,5) -- (8,0);
				\draw (1.6,2) -- (3.2,0);
				\draw (6,2.5) -- (4,0);
				\draw (7,1.25) -- (6,0);
				\filldraw [black] (0,0) circle (1pt);
				\filldraw [black] (8,0) circle (1pt);
				\filldraw [black] (3.2,0) circle (1pt);
				\filldraw [black] (4,0) circle (1pt);
				\filldraw [black] (6,0) circle (1pt);
				\node [below] at (0,0) {$A$};
				\node [below] at (3.2,0) {$B$};
				\node [below] at (4,0) {$C$};
				\node [below] at (6,0) {$D$};
				\node [below] at (8,0) {$E$};
				\node [above left] at (3,3.75) {$12$};
				\node [above left] at (1,1.25) {$8$};
				\node [above right] at (2.2,1.25) {$8$};
				\node [above right] at (5,3.75) {$10$};
				\node [above left] at (5,1.25) {$10$};
				\node [above] at (7,1.25) {$5$};
				\node [left] at (6.4,0.5) {$5$};
				\node [right] at (7.6,0.5) {$5$};
			\end{tikzpicture}
			\caption{Example of an equidistant tree.  This tree is also an ultrametric.}
			\label{fig:eq_tree}
	\end{figure}

For the tree in Figure \ref{fig:eq_tree}, there are two ways to express this tree:
\begin{enumerate}[1.]
\item As a tree ultrametric in $\mathcal{U}_5$: $(16, 40, 40, 40, 40, 40, 40, 20, 20, 10)$
\item As a vector in an ambient space, in terms of lengths of internal edges:\\
$(0, \ldots, 0, 12, 10, 5, 0, \ldots, 0)$, where the values of the nonzero coordinates in this vector are the internal edge lengths leading to clades $\{A,B\}$, $\{C,D,E\}$ and $\{D,E\}$.
\end{enumerate}
\end{example}

In general, using the internal edges to represent trees allows for an iterative construction of a family of clades satisfying one of the nested set conditions (\ref{eq:nested}).  In the case of Figure \ref{fig:eq_tree}, this family is $\{A,B\}$, $\{C, D, E\}$ and $\{D, E\}$.

The following lemma provides intuition on the definition of full dimensionality given above in Definition \ref{def:treetopology}.
\begin{lemma}\label{lem:dimub}
Let $N\ge 2$ be an integer.  For any tree topology $F$ on a ground set of $N$ elements, we have $|F|\le N-2$.
\end{lemma}

\begin{proof}
We proceed by induction on $N$.  When $N=2$, since $2 > N-1$, $F$ is necessarily empty because clades $S$ cannot satisfy $2 \leq |S| \leq 1$.  When $N=3$, all clades of $F$ must have cardinality $2$, because the cardinality is at least $2$ and at most $2$, and thus are among $\{1,2\},\{1,3\},\{2,3\}$.  But any two of these clades do not satisfy one of the nested set conditions (\ref{eq:nested}).  Hence $|F|\le 1$, and the base case for $N=2,3$ holds.
	
	Next, suppose Lemma \ref{lem:dimub} holds for $3\le N\le m$ where $m\ge 3$.  Consider the case when $N=m+1$: We take all clades $S$ belonging to $F$ that are maximal in terms of inclusion.  There are two cases:
	\begin{enumerate}[(i)]
	\item There is a unique maximal clade $S_{\max}$ in $F$: $|S_{\max}|\le m$ and all clades of $F$ are subsets of $S_{\max}$.  If $|S_{\max}|\le 2$, then $F$ has a unique clade, which is $S_{\max}$ and therefore $|F|=1\le m-1=N-2$; otherwise $|S_{\max}|\ge 3$.  Consider the family $F \backslash \{S_{\max}\}$ on the ground set $S_{\max}$.  For any $S\in F \backslash \{S_{\max}\}$, since $S$ is a proper subset of $S_{\max}$, we have $2\le |S|\le |S_{\max}|-1$.  Since $F \backslash \{S_{\max}\}$ is still a nested set, it is a tree topology on $S_{\max}$.  By the induction hypothesis, $|F \backslash \{S_{\max}\}|\le |S_{\max}|-2\le m-2$.  Thus, $|F|=|F \backslash \{S_{\max}\}|+1\le m-1=N-2$ and Lemma \ref{lem:dimub} holds for $N=m+1$.
	
	\item There are at least two maximal clades in $F$: Let these clades be $S_{1},S_{2},\ldots,S_{k}$, with $k\ge 2$, then $2\le |S_{i}|\le m$ for $1\le i\le k$.  Since $F$ is a nested set, these $S_{i}$ are pairwise disjoint.  So 
	\begin{equation*}
		\sum_{i=1}^{m}{|S_{i}|} \le m+1.
	\end{equation*}
	Let $c_{i}$ be the number of proper subsets of $S_{i}$ that belong to $F$.  Then $|F|=k+\sum_{i=1}^{k}{c_{i}}$.  Notice that $c_{i}$ is also the cardinality of a tree topology on $S_{i}$, so by the induction hypothesis, $c_{i}\le |S_{i}|-2$.  Hence
	\begin{equation}\label{eq:transition}
	|F| = k+\sum_{i=1}^{k}{c_{i}} \le \sum_{i=1}^{m}{|S_{i}|} - k \le (m+1)-2=N-2.
	\end{equation}
	Lemma \ref{lem:dimub} thus holds for $N=m+1$.  
	\end{enumerate}
This concludes the transition step, and the proof.
\end{proof}

Conversely, now, we consider the minimal clades of a tree topology $F$.

\begin{lemma}\label{lem:minelement}
Let $F$ be a tree topology on $[N]$.  For any pair $p\in p_{N}$, let $F(p)=\{S\in F\mid S\supseteq p\}$.  For $p\in p_{N}$, if $F(p)\ne \emptyset$, then the intersection of all clades $S\in F(p)$ is also an element of $F(p)$.
\end{lemma}

\begin{proof}
	Suppose $p\in p_{N}$ and $F(p)\ne \emptyset$. For any two clades $S_{1},S_{2}\in F(p)$, since $p\subseteq S_{1},S_{2}$, then $S_{1}$ and $S_{2}$ cannot be disjoint.  Thus, by the nested set condition in Definition \ref{def:treetopology}, either $S_{1}$ contains $S_{2}$ or vice versa.  This means that all clades in $F(p)$ form a completely ordered set with respect to set inclusion.  Since $F(p)$ is finite, it has a minimal element that must be contained in all other elements.  This minimal element is the intersection of all clades in $F(p)$.
\end{proof}

\begin{definition}\label{def:minelement}
The minimal element that is the intersection of all clades $S \in F(p)$ given in Lemma \ref{lem:minelement} is called the {\em closure} of $p$ in $F$.  We denote this closure by $\mathrm{cl}_{F}(p)$.  If $F(p)=\emptyset$, we set $\mathrm{cl}_{F}(p)$ to be $[N]$.
\end{definition}

\begin{example}
Let $F=\{1,2,3,4,5,6\}$ and $S=\big\{ \{1,2,3\}, \{1,2\}, \{1\}, \{5,6\} \big\}$.  Then $\mathrm{cl}_F(2) = \{1,2 \}$, $\mathrm{cl}_F(5) = \{5,6\}$, and $\mathrm{cl}_F(4) = F.$
\end{example}

Lemma \ref{lem:minelement} intuitively gives us the result that $\mathrm{cl}_{F}(p)$ is the minimal subset in $F \cup [N]$ that contains $p$.\\

Given the definition of a tree topology $F$, and notions of maximal and minimal clades of $F$, we now proceed to study the behavior of varying tree topologies.  We shall construct the setting for such a study via the definitions of an equivalence relation and a partial order on a tree topology $F$ in terms of pairs $p \in p_{N}$ as follows.

\begin{definition}\label{def:po}
Let $F$ be a tree topology on $[N]$. We define an equivalence relation $=_{F}$ on $p_N$ by
$$
p_1 =_F p_2 \mbox{~~if~~} \mathrm{cl}_{F}(p_1) = \mathrm{cl}_{F}(p_2),
$$
and a partial order $<_{F}$ on $p_{N}$ by
$$
p_{1} <_{F} p_{2} \mbox{~~if~~} \mathrm{cl}_{F}(p_1) \subsetneq \mathrm{cl}_{F}(p_2),
$$
for all pairs $p_1, p_2 \in p_N$.
\end{definition}

%By the definition of the equivalence relation $=_F$, every equivalence class with respect to $=_F$ corresponds to an element of $F$ or $[N]$.  Non-disjoint pairs $p_N$ are always comparable under $=_F$ or $<_F$.

\begin{lemma}\label{lem:comparable}
Let $F$ be a tree topology on $[N]$.  For any distinct elements $i,j,k\in [N]$, exactly one of the following holds:
\begin{align*}
\{i,j\} & =_{F} \{i,k\},\\
\{i,j\} & <_{F} \{i,k\},\\
\{i,k\} & <_{F} \{i,j\}.
\end{align*}	
\end{lemma}

\begin{proof}
Suppose for contradiction that Lemma \ref{lem:comparable} does not hold for some distinct $i,j,k$.  Then $F(\{i,j\})$ and $F(\{i,k\} )$ do not contain each other and there exist $S_{1},S_{2}\in F$ such that $S_{1} \in F(\{i,j\}) \backslash F(\{i,k\})$ and $S_{2} \in F(\{i,k\}) \backslash F(\{i,j\})$.  Thus $j\in S_{1} \backslash S_{2}$ and $k\in S_{2} \backslash S_{1}$, contradicting that $F$ is a nested set.  Hence Lemma \ref{lem:comparable} holds.
\end{proof}

Given this framework to compare pairs of elements, we now give two notions of bipartitioning of trees that are closely related.  It turns out, as we will see in Theorem \ref{thm:fulldim}, in the context of tree topologies, they coincide.  

\begin{definition}\label{def:binary}
	A rooted phylogenetic tree $T$ is said to be {\em binary} if every vertex of $T$ is either a leaf or trivalent.
\end{definition}

\begin{definition}\label{def:bifur}
	Let $F$ be a tree topology on $[N]$ and $\bar{F}= \big\{S\in F\mid |S|\ge 3 \big\} \cup [N]$.  $F$ is said to be {\em bifurcated} if for every $S\in \bar{F}$, exactly one of the following holds:
	\begin{enumerate}[(a)]
		\item there exists a proper subset $S' \subset S$ such that $S'\in F$ and $|S'|=|S|-1$; or
		\item there exist two proper subsets $S',S'' \subset S$ such that $S',S''\in F$ and $S'\cup S''=S$.
	\end{enumerate}
	Note that in (b), we must have that $S'\cap S''=\emptyset$.
\end{definition}

A binary tree is a bifurcating tree that has exactly two descendants stemming from each interior node.

\begin{theorem}\label{thm:fulldim}
	Let $F$ be a tree topology on $[N]$.  The following are equivalent:
	\begin{enumerate}[(1)]
		\item $F$ is full dimensional;
		\item for $1\le i<j<k \le N$, two of the pairs $\{i,j\},\{i,k\},\{j,k\}$ are $=_{F}$, and the third pair is $<_{F}$ than the other two $F$-equivalent pairs;
		\item $F$ is bifurcated;
		\item every phylogenetic tree with tree topology $F$ is binary.
	\end{enumerate}
\end{theorem}

\begin{proof}
	(3) $\Rightarrow$ (2): Suppose $F$ is bifurcated and consider distinct elements $i,j,k \in [N]$.  By Lemma \ref{lem:comparable}, any two of $\{i,j\},\{i,k\},\{j,k\}$ are comparable with respect to $=_{F}$ or $<_{F}$.  If the three pairs are all $=_{F}$, then by Definition \ref{def:po}, their closures in $F$ are equal.  Let this closure be $S\in F \cup \{[N]\}$, then $i,j,k\in S$ and $|S|\ge 3$, thus $S\in \bar{F}$.  Since $F$ is bifurcated, condition (a) or (b) in Definition \ref{def:bifur} holds for $S$.  If (a) holds, then there exists $S'\in F$ such that $S'$ is a proper subset of $S$ with $|S'|=|S|-1$.  In this case, at least two of $i,j,k$ belong to $S'$, and the closure of the pair formed by these elements is contained in $S'$---a contradiction.  So we may assume $\{i,k\} <_{F} \{i,j\}$.  
	
	We now need to show $F$-equivalence between $\{i,j\}$ and $\{j,k\}$: by Definition \ref{def:po}, there exists $S_{1}\in F$ such that $\{i,k\} \subseteq S_{1}$ but $\{i,j\} \not \subseteq S_{1}$.  Then $i,k\in S_{1}$ and $j\notin S_{1}$.  Now for any $S_{2}\in F$, if $\{j,k\} \subseteq S_{2}$, then $S_{2}$ and $S_{1}$ are not disjoint.  Since $j\in S_{2} \backslash S_{1}$, we must have $S_{1} \subseteq S_{2}$.  Then $i\in S_{2}$, and $\{i,k\} \subseteq S_{2}$.  Since $\{j,k\} \not \subseteq S_{1}$, by definition, $\{i,k\} <_{F} \{j,k\}$.  In addition, if an element of $F$ is a superset of $\{i,j\}$, then it also contains $k$ and thus is also a superset of $\{j,k\}$.  Conversely, being a superset of $\{j,k\}$ implies that it also contains $i$, and thus is also a superset of $\{i,j\}$.  Therefore $\{i,j\} =_{F} \{j,k\}$, and (2) holds.\\
	
	(2) $\Rightarrow$ (3): Suppose (2) holds for $F$.  We will show that $F$ is bifurcated: for any subset $S\in \bar{F}$, consider all maximal proper subsets $M_{1},\ldots,M_{m}$ of $S$ that are clades in $F$.  For any two such maximal subsets, since neither can be a subset of the other by definition, they must be disjoint.  If $m\ge 3$, we can choose $i,j,k\in [N]$ from $M_{1},M_{2},M_{3}$ respectively.  Then $\mathrm{cl}_{F}\big(\{i,j\}\big)=\mathrm{cl}_{F}\big(\{i,k\} \big)=\mathrm{cl}_{F}\big(\{j,k\} \big)=S$ and thus $\{i,j\} =_{F} \{i,k\} =_{F} \{j,k\}$---a contradiction.  Therefore, $m$ must be either $1$ or $2$.  
	
	If $m=1$, then $2\le |M_{1}|\le |S|-1$.  Suppose $|M_{1}|\le |S|-2$, we can choose two elements $i,j\in S \backslash M_{1}$ and another element $k\in M_{1}$.  Then we also have $\mathrm{cl}_{F}(\{i,j\})=\mathrm{cl}_{F}(\{i,k\})=\mathrm{cl}_{F}(\{j,k\})=S$ and $\{i,j\} =_{F} \{i,k\} =_{F} \{j,k\}$---a contradiction.  Hence $|M_{1}|=|S|-1$ and condition (a) holds.  If $m=2$, we already have disjoint $M_{1}, M_{2}\in F$ that are proper subsets of $S$.  If $M_{1}\cup M_{2} \ne S$, we may choose $i\in S \backslash (M_{1}\cup M_{2})$ and $j\in M_{1},k\in M_{2}$.  Then $\mathrm{cl}_{F}(\{i,j\})=\mathrm{cl}_{F}(\{i,k\})=\mathrm{cl}_{F}(\{j,k\})=S$ and $\{i,j\} =_{F} \{i,k\} =_{F} \{j,k\}$---a contradiction.  Hence $M_{1}\cup M_{2} = S$ and condition (b) holds.  Hence $F$ is bifurcated.\\
	
	(1) $\Rightarrow$ (3): We proceed by induction on $N$.  When $N=3$ and $F$ is full dimensional, then $|F|=1$ and $F$ consists of one $2$-element subset.  Thus, condition (a) holds for $\{1,2,3\}$ and $F$ is bifurcated.  
	
	Suppose (1) $\Rightarrow$ (3) holds for eligible integers less than $N$.  Let $F$ be a full-dimensional nested set on $[N]$, then $|F|=N-2$.  Consider the maximal elements $S_{1},\ldots,S_{k}\in F$, $k\ge 1$, with respect to set inclusion.  Then by the case (ii) in the proof of Lemma \ref{lem:dimub}, $|F|\le N-k$.  Hence $k\le 2$.  
	
	If $k=1$, there exists a unique maximal element $S\in F$ with $|S|\le N-1$.  Then $F \backslash \{S\}$ is a nested set on the ground set $S$, and $N-3=|F \backslash \{S\}|\le |S|-2\le N-1-2=N-3$.  So $|S|=N-1$, and condition (a) holds for $[N]$.  In addition, $|F \backslash \{S\}| = |S|-2$, so $F \backslash \{S\}$ is full dimensional.  By the induction hypothesis, all elements in $\bar{F} \backslash \{S\}$ satisfy either condition (a) or (b).  Note that $\bar{F}=\bar{F} \backslash \{S\} \cup \{[N]\}$, so (3) also holds for $F$.  
	
	If $k=2$, all equalities hold in (\ref{eq:transition}), so there are two maximal elements $S_{1},S_{2}\in F$ with $S_{1}\cap S_{2}=\emptyset$ and $|S_{1}|+|S_{2}|=N$.  So condition (b) holds for $[N]$.  Let $F_{i}$ be the set of the proper subsets of $S_{i}$ that belong to $F$ for $i=1,2$.  Then both $F_{i}$ are full-dimensional tree topologies on their respective ground sets $S_{i}$.  By the induction hypothesis, both $F_{i}$ are bifurcated and all elements in $\bar{F}_{i}$ satisfy either condition (a) or (b).  Note that $\bar{F}=F \cup \{[N]\}=\bar{F}_{1}\cup \bar{F}_{2} \cup \{[N]\}$, hence $F$ is also bifurcated, which completes the transition step.\\
	
	(3) $\Rightarrow$ (1): We proceed by induction on $N$.  When $N=3$ and $F$ is bifurcated, then condition (a) holds for $\{1,2,3\}$ and $F$ only contains one $2$-element subset, so $F$ is full dimensional. 
	
	Suppose (3) $\Rightarrow$ (1) holds for all eligible integers less than $N$.  For any bifurcated nested set $F$ on $[N]$, either condition (a) or (b) holds for the set $[N]$.  If (a) holds, then there exists $S\in F$ such that $|S|=N-1$.  Then all elements in $F \backslash \{S\}$ are proper subsets of $S$ and they thus form a nested set on the ground set $S$.  This nested set is also bifurcated, by the induction hypothesis, so it is full dimensional.  So $|F \backslash \{S\}|=|S|-1=N-2$ and $|F|=|F \backslash \{S\}|+1=N-1$, $F$ is full dimensional.  
	
	If condition (b) holds, then there exist disjoint $S_{1},S_{2}\in F$ such that $S_{1} \cup S_{2} = [N]$.  Let $F_{i}$ be the elements of $F$ that are proper subsets of $S_{i}$ for $i=1,2$. Then $F=F_{1} \cup F_{2} \cup \{S_{1},S_{2}\}$.  Each $F_{i}$ is a nested set on the ground set $S_{i}$ (and may be empty when $|S_{i}|=2$); it is still bifurcated.  By the induction hypothesis, $F_{i}$ is full dimensional and $|F_{i}|=|S_{i}|-2$.  Then $|F|=(|S_{1}|-2)+(|S_{2}|-2)+2=N-2$, so $F$ is still full dimensional. This completes the transition step.\\
	
	(3) $\Rightarrow$ (4): Suppose $F$ is bifurcated and a rooted phylogenetic tree $T$ has tree topology $F$.  Let $v$ be a non-leaf node.  It suffices to show that $v$ has degree $3$. We consider two cases: 
	
	(i) Suppose $v$ is not the root of $T$.  Then there is a unique path from the root of $T$ to $v$.  Along the path, there is an edge connecting $v$, and this edge corresponds to a clade $S$ in $F$.  Since $F$ is bifurcated, $S$ satisfies either conditions in Definition \ref{def:binary}.  If there exists a proper subset $S'$ such that $S'$ is also a clade of $F$ and $|S'|=|S|-1$, then all other edges connecting $v$ include one edge connecting to the leaf in $S \backslash S'$ and one edge corresponding to $S'$.  Otherwise there exist clades $S',S''$ of $F$ such that $S'\cup S''=S$, then all other edges connecting $v$ include the two edges corresponding to $S'$ and $S''$.  In either case, $v$ is trivalent.
	
	(ii) Suppose $v$ is the root of $T$.  Then $v$ is connected to the virtual leaf $0$.  Since $[N]\in \bar{F}$, $[N]$ satisfies either conditions in Definition \ref{def:binary} and $v$ is connecting to two other edges in either case, so $v$ is also trivalent.\\
	
	(4) $\Rightarrow$ (3): Suppose a binary rooted tree $T$ has tree topology $F$.  For any clade $S\in \bar{F}$, $S$ corresponds to an edge $e$ of $T$.  Let $v$ be the vertex of $e$ with greater distance to the root of $T$.  Since $T$ is binary, $v$ is trivalent and connects to other two edges $e'$ and $e''$.  Each leaf in $S$ has a unique path to $v$, which must contain $e'$ or $e''$.  This admits a partition of $S$ into two nonempty subsets.  If both subsets have cardinality of at least $2$, then they are both clades of $F$, and we have $S'\cap S''=S$.  Otherwise one of them is a singleton, say $S''$, and thus $S'$ is a clade of $F$ with $|S'|=|S|-1$.  Thus $S$ satisfies the condition in Definition \ref{def:binary}.  For $[N]\in \bar{F}$, since $T$ is binary, the root of $T$ is also connected by two edges other than the edge to the virtual leaf.  The reasoning above applies to $[N]$, and $F$ is bifurcated.
\end{proof}

These concepts of bipartitioning in terms of tree topologies are important in understanding the combinatorial aspects of tree ultrametrics.  In addition to the fact that tree ultrametrics are equidistant trees, it is also true that every point along any tropical line segment between two equidistant trees is also itself an equidistant tree.  The equivalence relation $=_F$ and partial order $<_F$ completely define the set of all ultrametrics for a given tree topology $F$, which we present and formalize in the following results.

\begin{proposition}\label{prop:conedefn}
	Given a tree topology $F$ on $[N]$, let $ut(F)$ be the set of all ultrametrics in $\mathcal{U}_{N}$ corresponding to a tree with tree topology $F$. Then 	
	\begin{equation}\label{eq:w-po}
		ut(F) = \big\{(w_{p})_{p\in p_{N}}\in \mathcal{U}_{N} \mid w_{p_{1}}=w_{p_{2}} \text{ if } p_{1} =_{F} p_{2} \text{ and } w_{p_{1}} < w_{p_{2}} \text{ if } p_{1} <_{F} p_{2} \big\}.
	\end{equation}
\end{proposition}

\begin{proof}
	Fix a tree topology $F$ on $[N]$ and a corresponding equidistant tree $T$.  For each internal edge of $T$ indexed by a clade $S$, let $\ell(S)>0$ be its length.  For the external edge of $T$, connecting the $i$th leaf to the root of $T$, let $\ell_{i}$ be its length.  Let $h$ be the height of $T$.  Then, by definition, the distance of each leaf to the root of $T$ is $h$.  There exists a unique path from the $i$th leaf to the root, consisting of the external edge connecting the $i$th leaf, and some internal edges: an internal edge indexed by $S$ appears on this path if and only if the $i$th leaf and the root are separated by the internal edge itself.  This necessarily means that $i\in S$.  Then 	
	\begin{equation}\label{eq:height}
		h = \ell_{i} + \sum_{i\in S}{\ell(S)}.
	\end{equation}
	
	Next we consider the path connecting the $i$th and $j$th leaves.  This path consists of the two external edges and some internal edges.  An internal edge indexed by $S$ appears on this path if and only if the $i$th leaf and the $j$th leaf are separated by the edge.  Equivalently, this means that $|\{i,j\}\cap S|=1$.  Hence
	\begin{equation*}\label{eq:path}
		\begin{split}
			w_{\{i,j\}}&=\ell_{i} + \ell_{j} \sum_{|\{i,j\}\cap S|=1}{\ell(S)} \\
					&=2h-\sum_{i\in S}{\ell(S)}-\sum_{j\in S}{\ell(S)} + \sum_{|\{i,j\}\cap S|=1}{\ell(S)} \\
					&=2h-2\sum_{i,j\in S}{\ell(S)} \\
					&=2h-2\sum_{S\in F(\{i,j\})}{\ell(S)}.
		\end{split}
	\end{equation*}
Therefore $w$ satisfies the condition defining the set (\ref{eq:w-po}).
	
	Conversely, suppose a vector $w\in \mathbb{R}^{n}/\mathbb{R}\one$ satisfies the conditions defining the set (\ref{eq:w-po}). Then the system of linear equations
$$
w_{\{i,j\}} = 2h - 2\sum_{S\in F(\{i,j\})}{x_{S}} \quad \forall \ \{i,j\}\in p_{N}
$$
has a solution such that $h\in \mathbb{R}$ and $x_{S}>0$ for all $S\in F$.  By a translation in $\mathbb{R}$, we may choose a sufficiently large $h$ such that all $\ell_{i}$ in (\ref{eq:height}) are positive.  Then $w$ is the ultrametric of an equidistant tree with external edge lengths $\ell_{i}$ and internal edge lengths $x_{S}$, whose tree topology is $F$.
\end{proof}

\begin{corollary}\label{cor:order}
	Let $F$ be a tree topology on $[N]$ and $w\in ut(F)$.  Then for any pairs $p,q\in p_{N}$, if $p\cap q \ne \emptyset$, then $w_{p}=w_{q}$ implies $p =_{F} q$ and $w_{p} < w_{q}$ implies $p <_{F} q$.
\end{corollary}

\begin{proof}
	If $p=q$, then $w_{p}=w_{q}$ holds and $p =_{F} q$ also holds.  Otherwise we may assume $p=\{i,j\}$ and $q=\{i,k\}$.
	By Lemma \ref{lem:comparable}, one of the following relationships holds: $p <_{F} q,$ or $p =_{F} q,$ or $q<_{F} p$.  By Proposition \ref{prop:conedefn}, these three relationships imply $w_{p}<w_{q},$ or $w_{p}=w_{q},$ or $w_{p}>w_{q}$ respectively. Hence the converse implications also hold.
\end{proof}

So far, we have introduced means to study tree structures by studying subsets of leaves and iteratively dividing and comparing these subsets.  We have also determined the set of tree ultrametrics defined by the comparison framework set up in Definition \ref{def:po}.  Given this framework, we now determine when we have compatibility of sets in order to present our combinatorial result that gives all possible tree topologies along a tropical line segment.

We note that the geometric and combinatorial procedure we present is a natural approach that has been previously implemented in other tree settings and, more generally, in finite metric spaces, e.g., by \cite{BANDELT199247,DRESS1984321}.  Our approach differs in two important ways: firstly, we consider the infinite space of all sets of phylogenetic trees, and secondly, our study is fundamentally tropical, since we use the framework of tropical line segments.

\begin{definition}\label{def:ttcompatible}
	Let $F_{1}$ and $F_{2}$ be tree topologies on $[N]$.  Define the set of {\em compatible} tree topologies of $F_{1}, F_{2}$ to consist of tree topologies $F$ where there exist tree ultrametrics $w_{1}\in ut(F_{1})$ and $w_{2}\in ut(F_{2})$, such that the tree ultrametric $w_{1} \boxplus w_{2} \in ut(F)$.  We denote this set by $C(F_{1},F_{2})$.
\end{definition}

\begin{theorem}\label{thm:compatible}
Let $F_{1}, F_{2}, F$ be full-dimensional tree topologies on $[N]$.  If $F\in C(F_{1},F_{2})$, then any equivalence class $C\subseteq p_{N}$ with respect to $=_{F}$ is contained in an equivalence class with respect to either $=_{F_{1}}$ or $=_{F_{2}}$.  
	
Put differently, for each $S\in \bar{F}$, there exists either $S_{1}\in F_{1}$ such that for $p\in p_{N}$, if $\mathrm{cl}_{F}(p)=S$, then $\mathrm{cl}_{F_{1}}(p)=S_{1}$; or $S_{2}\in F_{2}$ such that for $p\in p_{N}$, if $\mathrm{cl}_{F}(p)=S$, then $\mathrm{cl}_{F_{2}}(p)=S_{2}$.
\end{theorem}

\begin{proof}
	Suppose $F_{3}\in C(F_{1},F_{2})$. There exist ultrametrics $w^{1},w^{2},w^{3}$ such that for $p\in p_{N}$, $w^{3}_{p}=\max(w^{1}_{p},w^{2}_{p})$ and $w^{i}\in ut(F_{i})$ for $i=1,2,3$.  Choose $S\in \bar{F} = \big\{S\in F\mid |S|\ge 3 \big\} \cup [N]$.  By Theorem \ref{thm:fulldim}, $F$ is bifurcated, so condition (a) or (b) of Definition \ref{def:bifur} holds for $S$.  If (a) holds, we set $X=S'$ and $Y=S \backslash S'$; if (b) holds, we set $X=S_{1}$ and $Y=S_{2}$.  Then for $p=\{i,j\}\in p_{N}$, $\mathrm{cl}_{F}(p)=S$ if and only if $(i,j)\in X\times Y$ or $(j,i)\in X\times Y$.  Let $M=w^{3}_{\{ i,j \}}$ for all $i\in X$ and $j\in Y$, then
	\begin{equation}\label{eq:mmax}
		\max(w^{1}_{\{i,j\}},w^{2}_{\{i,j\}})=M 
	\end{equation}
for all $i\in X$ and $j\in Y$.
	
	Consider a complete, bipartite graph $G := K_{|X|,|Y|}$ with vertices $X\cup Y$.  Recall that the vertices of a bipartite graph can be partitioned into two disjoint, independent sets where every graph edge connects a vertex in one set to one in the other.  Thus, for $(i,j)\in X\times Y$, if $w^{1}_{\{i,j\}}=M$, then we call the edge $(i,j)$ of $G$ pink; if $w^{2}_{\{i,j\}}=M$, then we call the edge $(i,j)$ of $G$ purple.  Then, each edge in $G$ is pink, purple, or both pink and purple.  We claim that in fact, either all edges of $G$ are pink, or all edges of $G$ are purple. 
	
	Suppose there exists a non-pink edge $(i,j)\in X\times Y$ in $G$.  Then this edge is purple: $M=w^{2}_{\{i,j\}}\ne w^{1}_{\{i,j\}}$.  By (\ref{eq:mmax}), $w^{1}_{\{i,j\}} < M$.  For any other element $j'\in Y,j'\ne j$, if $(i,j')$ is not purple, then $w^{2}_{\{i,j'\}}<M=w^{1}_{\{i,j'\}}$.  Note that $w^{1}_{\{i,j\}} < w^{1}_{\{i,j'\}} $, thus by Corollary \ref{cor:order}, $\{i,j\} <_{F_{1}} \{i,j'\}$.  Similarly, since $w^{2}_{\{i,j'\}} < w^{2}_{\{i,j\}}$, we have $\{i,j'\} <_{F_{2}} \{i,j\}$.  By Theorem \ref{thm:fulldim}, $\{j,j'\}=_{F_{1}} \{i,j'\}$ and $\{j,j'\}=_{F_{2}} \{i,j\}$.  Then $w^{1}_{\{j,j'\}}=w^{2}_{\{j,j'\}}=M$ and $w_{\{j,j'\}}=M$.  By Corollary \ref{cor:order}, this means that $\{i,j\} =_{F} \{i,j'\} =_{F} \{j,j'\}$, which contradicts that $F$ is full dimensional.  Therefore $(i,j')$ must be purple.  Symmetrically, for any $i'\in X$ with $i'\ne i$, $(i',j)$ must be purple.  Then for such $i'$ and $j'$, we have
	\begin{equation*}
		w^{2}_{\{i,j\}} = w^{2}_{\{i,j'\}} = w^{2}_{\{i',j\}} = M.
	\end{equation*}
By Corollary \ref{cor:order}, we have $\{i,j\} =_{F_{2}} \{i,j'\} =_{F_{2}} \{i',j\}$.  Since $F_{2}$ is full dimensional, Theorem \ref{thm:fulldim} gives us that $\{j,j'\} <_{F_{2}} \{i,j\} =_{F_{2}} \{i,j'\} $.  Then $\{j,j'\} <_{F_{2}} \{i',j\}$, and thus $\{i',j'\} =_{F_{2}} \{i',j\}$.  Hence $w^{2}_{\{i',j'\}}=M$ and the edge $(i',j')$ in $G$ is purple too.  So all edges in $G$ are purple, thus proving our claim.
	
	Finally, if all edges in $G$ are pink, then all $w^{1}_{\{i,j\}}$ are equal for $(i,j)\in X\times Y$.  By Corollary \ref{cor:order}, all these $\{i,j\}$ belong to the same equivalence class with respect to $=_{F_{1}}$; symmetrically, if all edges in $G$ are purple, then all $w^{2}_{\{i,j\}}$ are equal for $(i,j)\in X\times Y$ and all these $\{i,j\}$ belong to the same equivalence class with respect to $=_{F_{2}}$.
\end{proof}

There are two important remarks concerning this result to highlight:

\begin{remark}\label{rem:nonfull}
Theorem \ref{thm:compatible} may still hold when one of the tree topologies is not full dimensional, say $F_1$.  For example, let $N=5$ and  
	\begin{equation*}
		F_{1}=\{3,4\}, \mbox{~~~} F_{2}= \big\{\{1,4\}, \{2,3\}, \{1,2,3,4\} \big\}.
	\end{equation*}
For $F= \big\{\{1,4\}, \{1,3,4\}, \{2,5\} \big\}$, $F\notin C(F_{1},F_{2})$: Suppose $w^{1}\in ut(F_{1})$ and $w^{2}\in ut(F_{2})$.  Since $\mathrm{cl}_{F_{i}}(\{2,5\})=[5]$ for $i=1,2$, by definition $(w^{i})_{\{2,5\}}=\max_{p\in p_{5}}{(w^{i})_{p}}$ for $i=1,2$.  Let $w=w^{1} \boxplus w^{2}$.  Then $w_{\{2,5\}}=\max_{p\in p_{5}}{w_{p}}$ also.  However, since $\{2,5\}\in F$, $\mathrm{cl}_{F}\big(\{2,5\} \big) \ne [5]$ and thus $w\notin ut(F)$. 
	Nevertheless, the conclusion of Theorem \ref{thm:compatible} holds for $F$, since for $p\in p_{5}$, we have:
$$
	\begin{array}{ccccccc}
		\mathrm{cl}_{F}(p)&=&[5] & \mbox{~implies~} & \mathrm{cl}_{F_{1}}(p)&=&[5];\\
		\mathrm{cl}_{F}(p)&=&\{1,3,4\} & \mbox{~implies~} & \mathrm{cl}_{F_{2}}(p)&=&\{1,2,3,4\};\\
		\mathrm{cl}_{F}(p)&=&\{1,4\} & \mbox{~implies~} & \mathrm{cl}_{F_{2}}(p)&=&\{1,4\};\\
		\mathrm{cl}_{F}(p)&=&\{2,5\} & \mbox{~implies~} & \mathrm{cl}_{F_{1}}(p)&=&[5].\\
	\end{array}
$$
\end{remark}

\begin{remark}\label{rem:compatible}
Theorem \ref{thm:compatible} is a combinatorial approach to searching for all possible tree topologies a tropical line segment of tree ultrametrics can traverse; Example \ref{ex:compatible} illustrates this procedure.  However, the converse is not true: given full-dimensional tree topologies satisfying the conditions of Theorem \ref{thm:compatible}, it is not always possible to construct tropical line segments that traverse these topologies; see Example \ref{ex:12leaves} for details.
\end{remark}

\begin{example}\label{ex:compatible}
Let $N=5$ and choose the following two full-dimensional tree topologies on $[5]$: 
	\[F_{1} = \big\{\{1,2,3\},\{1,2\},\{4,5\} \big\} \mbox{~~and~~} F_{2} = \big\{\{1,3,4,5\},\{1,3,5\},\{1,5\} \big\}.\] 
Then the full-dimensional tree topologies in $C(F_{1},F_{2})$ are $F_{1},F_{2}$ themselves, and three others:
	\[ \big\{\{1, 3, 4, 5\}, \{1, 3, 5\}, \{1, 3\} \big\}, \mbox{~~~}
	\big\{\{1, 3, 4, 5\}, \{4, 5\}, \{1, 3\} \big\}, \mbox{~~~}
	\big\{\{1, 2, 3\}, \{1, 3\}, \{4, 5\} \big\}.\]
Note that $F_{3}= \big\{\{1, 2, 3\}, \{2, 3\}, \{4, 5\} \big\}$ does not belong to $C(F_{1},F_{2})$, because $\{1,2\}$ and $\{1,3\}$ form an equivalence class with respect to $=_{F}$, but $\{1,3\} <_{F_{i}} \{1,2\}$ for $i=1,2$.  This gives an example of the non-existence of certain types of trees.
\end{example}

We now give an example where the converse direction of Theorem \ref{thm:compatible} fails.

\begin{example}\label{ex:12leaves}

\begin{figure}[h]
	\centering
	\begin{minipage}{0.3\textwidth}
		\centering
		\begin{tikzpicture}[scale=0.225]
		\draw (0,0) -- (1.5,3) -- (3,0);
		\draw (2,0) -- (1,2);
		\draw (4,0) -- (5.5,3) -- (7,0);
		\draw (5,0) -- (6,2);
		\draw (1.5,3) -- (3.5,7) -- (5.5,3);
		\draw (8,0) -- (9.5,3) -- (11,0);
		\draw (10,0) -- (9,2);
		\draw (12,0) -- (13.5,3) -- (15,0);
		\draw (13,0) -- (14,2);
		\draw (9.5,3) -- (11.5,7) -- (13.5,3);
		\draw (3.5,7) -- (7.5,15) -- (11.5,7);
		\filldraw [black] (0,0) circle (1pt);
		\filldraw [black] (2,0) circle (1pt);
		\filldraw [black] (3,0) circle (1pt);
		\filldraw [black] (4,0) circle (1pt);
		\filldraw [black] (5,0) circle (1pt);
		\filldraw [black] (7,0) circle (1pt);
		\filldraw [black] (8,0) circle (1pt);
		\filldraw [black] (10,0) circle (1pt);
		\filldraw [black] (11,0) circle (1pt);
		\filldraw [black] (12,0) circle (1pt);
		\filldraw [black] (13,0) circle (1pt);
		\filldraw [black] (15,0) circle (1pt);
		\node [below] at (0,0) {\tiny{$1$}};
		\node [below] at (7,0) {\tiny{$2$}};	
		\node [below] at (8,0) {\tiny{$3$}};
		\node [below] at (15,0) {\tiny{$4$}};
		\node [below] at (10,0) {\tiny{$5$}};
		\node [below] at (13,0) {\tiny{$6$}};	
		\node [below] at (2,0) {\tiny{$7$}};
		\node [below] at (5,0) {\tiny{$8$}};
		\node [below] at (3,0) {\tiny{$9$}};
		\node [below] at (12,0) {\tiny{$10$}};	
		\node [below] at (11,0) {\tiny{$11$}};
		\node [below] at (4,0) {\tiny{$12$}};	
		\node [above] at (7.5,15) {$F_{1}$};
		\end{tikzpicture}
	\end{minipage}
	\begin{minipage}{0.3\textwidth}
		\centering
		\begin{tikzpicture}[scale=0.225]
		\draw (0,0) -- (1.5,3) -- (3,0);
		\draw (2,0) -- (1,2);
		\draw (4,0) -- (5.5,3) -- (7,0);
		\draw (6,0) -- (5,2);
		\draw (1.5,3) -- (3.5,7) -- (5.5,3);
		\draw (8,0) -- (9.5,3) -- (11,0);
		\draw (10,0) -- (9,2);
		\draw (12,0) -- (13.5,3) -- (15,0);
		\draw (14,0) -- (13,2);
		\draw (9.5,3) -- (11.5,7) -- (13.5,3);
		\draw (3.5,7) -- (7.5,15) -- (11.5,7);
		\filldraw [black] (0,0) circle (1pt);
		\filldraw [black] (2,0) circle (1pt);
		\filldraw [black] (3,0) circle (1pt);
		\filldraw [black] (4,0) circle (1pt);
		\filldraw [black] (6,0) circle (1pt);
		\filldraw [black] (7,0) circle (1pt);
		\filldraw [black] (8,0) circle (1pt);
		\filldraw [black] (10,0) circle (1pt);
		\filldraw [black] (11,0) circle (1pt);
		\filldraw [black] (12,0) circle (1pt);
		\filldraw [black] (14,0) circle (1pt);
		\filldraw [black] (15,0) circle (1pt);
		\node [below] at (0,0) {\tiny{$1$}};
		\node [below] at (3,0) {\tiny{$2$}};	
		\node [below] at (4,0) {\tiny{$3$}};
		\node [below] at (7,0) {\tiny{$4$}};
		\node [below] at (8,0) {\tiny{$5$}};
		\node [below] at (11,0) {\tiny{$6$}};	
		\node [below] at (12,0) {\tiny{$7$}};
		\node [below] at (15,0) {\tiny{$8$}};
		\node [below] at (2,0) {\tiny{$9$}};
		\node [below] at (6,0) {\tiny{$10$}};	
		\node [below] at (10,0) {\tiny{$11$}};
		\node [below] at (14,0) {\tiny{$12$}};
		\node [above] at (7.5,15) {$F$};
		\end{tikzpicture}
	\end{minipage}
	\begin{minipage}{0.3\textwidth}
		\centering
		\begin{tikzpicture}[scale=0.225]
		\draw (0,0) -- (3.5,7) -- (5.5,3);
		\draw (2,0) -- (4.5,5);
		\draw (3,0) -- (5,4);
		\draw (4,0) -- (5.5,3) -- (7,0);
		\draw (6,0) -- (5,2);
		\draw (8,0) -- (9.5,3) -- (11,0);
		\draw (10,0) -- (9,2);
		\draw (12,0) -- (13.5,3) -- (15,0);
		\draw (14,0) -- (13,2);
		\draw (9.5,3) -- (11.5,7) -- (13.5,3);
		\draw (3.5,7) -- (7.5,15) -- (11.5,7);
		\filldraw [black] (0,0) circle (1pt);
		\filldraw [black] (2,0) circle (1pt);
		\filldraw [black] (3,0) circle (1pt);
		\filldraw [black] (4,0) circle (1pt);
		\filldraw [black] (6,0) circle (1pt);
		\filldraw [black] (7,0) circle (1pt);
		\filldraw [black] (8,0) circle (1pt);
		\filldraw [black] (10,0) circle (1pt);
		\filldraw [black] (11,0) circle (1pt);
		\filldraw [black] (12,0) circle (1pt);
		\filldraw [black] (14,0) circle (1pt);
		\filldraw [black] (15,0) circle (1pt);
		\node [below] at (0,0) {\tiny{$1$}};
		\node [below] at (3,0) {\tiny{$2$}};	
		\node [below] at (4,0) {\tiny{$3$}};
		\node [below] at (7,0) {\tiny{$4$}};
		\node [below] at (8,0) {\tiny{$5$}};
		\node [below] at (14,0) {\tiny{$6$}};	
		\node [below] at (10,0) {\tiny{$7$}};
		\node [below] at (15,0) {\tiny{$8$}};
		\node [below] at (2,0) {\tiny{$9$}};
		\node [below] at (6,0) {\tiny{$10$}};	
		\node [below] at (12,0) {\tiny{$11$}};
		\node [below] at (11,0) {\tiny{$12$}};
		\node [above] at (7.5,15) {$F_{2}$};
		\end{tikzpicture}
	\end{minipage}
	\caption{Tree topologies $F_{1},F_{2},F$ with $12$ leaves in Example \ref{ex:12leaves}.}
\end{figure}
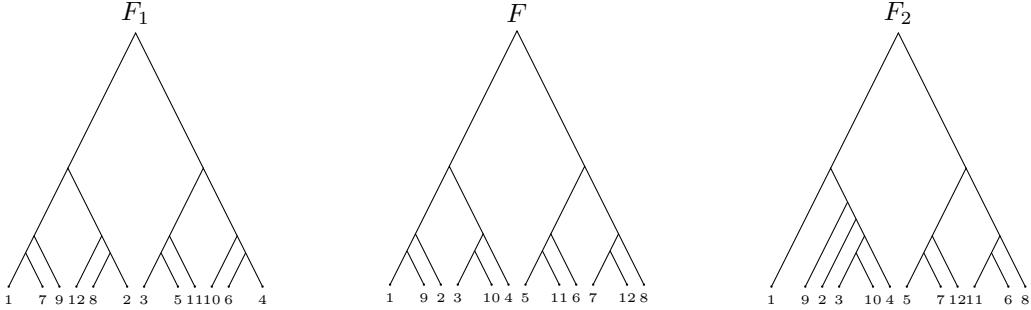

Following Remark \ref{rem:compatible}, let $N=12$.  Consider the following three tree topologies $F_{1},F_{2},F$, depicted in Figure \ref{fig:permutation2}:
\begin{align*}
F_{1}=\big\{&\{1,2,7,8,9,12\},\, \{1,7,9\},\, \{2,8,12\},\, \{1,7\},\, \{2,8\},\\
&\{3,4,5,6,10,11\},\, \{3,5,11\},\, \{4,6,10\},\, \{3,5\},\, \{4,6\} \big\} \\
F_{2} =\big\{&\{1,2,3,4,9,10\},\, \{2,3,4,9,10\},\, \{2,3,4,10\},\, \{3,4,10\},\, \{3,10\},\\
&\{5,6,7,8,11,12\},\, \{5,7,12\},\, \{6,8,11\},\, \{5,7\},\, \{6,8\} \big\} \\
F = \big\{&\{1,2,3,4,9,10\},\, \{1,2,9\},\, \{3,4,10\},\, \{1,9\},\, \{3,10\},\\
&\{5,6,7,8,11,12\},\, \{5,6,11\},\, \{7,8,12\},\, \{5,11\},\, \{7,12\}\big\}. \\
\end{align*}
Note that $|F_{1}|=|F_{2}|=|F|=10$, so $F_{1},F_{2},F$ are all full dimensional.  Moreover, Theorem \ref{thm:compatible} holds for $F$, since for $p\in p_{12}$, we have
$$ 	
	\begin{array}{ccccccccc}
		\mathrm{cl}_{F}(p)&=&[12] & \mbox{ implies } & \mathrm{cl}_{F_{2}}(p)&=&[12];\\
		\mathrm{cl}_{F}(p)&=&\{1,2,3,4,9,10\} & \mbox{ implies } & \mathrm{cl}_{F_{2}}(p)&=&\{1,2,3,4,9,10\};\\
		\mathrm{cl}_{F}(p)&=&\{5,6,7,8,11,12\} & \mbox{ implies } & \mathrm{cl}_{F_{2}}(p)&=&\{5,6,7,8,11,12\};\\
		\mathrm{cl}_{F}(p)&=&\{1,2,9\} & \mbox{ implies } & \mathrm{cl}_{F_{1}}(p)&=&\{1,2,7,8,9,12\};\\
		\mathrm{cl}_{F}(p)&=&\{3,4,10\} & \mbox{ implies } & \mathrm{cl}_{F_{2}}(p)&=&\{3,4,10\};\\
		\mathrm{cl}_{F}(p)&=&\{5,6,11\} & \mbox{ implies } & \mathrm{cl}_{F_{1}}(p)&=&\{3,4,5,6,10,11\};\\
		\mathrm{cl}_{F}(p)&=&\{7,8,12\} & \mbox{ implies } & \mathrm{cl}_{F_{2}}(p)&=&\{5,6,7,8,11,12\};\\
		\mathrm{cl}_{F}(p)&=&\{1,9\} & \mbox{ implies } & \mathrm{cl}_{F_{1}}(p)&=&\{1,7,9\};\\
		\mathrm{cl}_{F}(p)&=&\{3,10\} & \mbox{ implies } & \mathrm{cl}_{F_{2}}(p)&=&\{3,10\};\\
		\mathrm{cl}_{F}(p)&=&\{5,11\} & \mbox{ implies } & \mathrm{cl}_{F_{1}}(p)&=&\{3,5,11\};\\
		\mathrm{cl}_{F}(p)&=&\{7,12\} & \mbox{ implies } & \mathrm{cl}_{F_{2}}(p)&=&\{5,7,12\}.\\
	\end{array}
$$
However, $F\notin C(F_{1},F_{2})$: suppose there exist $w\in ut(F), w^{1}\in ut(F_{1}), w^{2}\in ut(F_{2})$ such that $w = w^{1} \boxplus w^{2}$. Note that $\{1,2\} =_{F} \{2,9\}$ and $\{1,2\} =_{F_{1}} \{2,9\}$ but $\{2,9\} <_{F_{2}} \{1,2\}$. Then 
	\begin{equation*}
		w^{2}_{\{2,9\}} < w^{2}_{\{1,2\}} \le \max\big(w^{1}_{\{1,2\}},w^{2}_{\{1,2\}} \big) = w_{\{1,2\}}  = w_{\{2,9\}}=\max\big(w^{1}_{\{2,9\}},w^{2}_{\{2,9\}} \big).
	\end{equation*}
	Hence $w^{1}_{\{2,9\}}>w^{2}_{\{2,9\}}$, and we have 
	\begin{equation*}
		w^{1}_{\{1,2\}} = w^{1}_{\{2,9\}} = w_{\{2,9\}}=w_{\{1,2\}}=\max\big(w^{1}_{\{1,2\}},w^{2}_{\{1,2\}} \big).
	\end{equation*}
Thus, 
	\begin{equation}\label{eq:12leaves-12}
		w^{1}_{\{1,2\}} \ge w^{2}_{\{1,2\}}.
	\end{equation}
Similarly, $\{4,10\} <_{F_{1}} \{3,4\}$ implies that
	\begin{equation}\label{eq:12leaves-34}
		w^{2}_{\{3,4\}} \ge w^{1}_{\{3,4\}}.
	\end{equation}
We also have $\{6,11\} <_{F_{2}} \{5,6\}$, which implies that
	\begin{equation}\label{eq:12leaves-56}
		w^{1}_{\{5,6\}} \ge w^{2}_{\{5,6\}},
	\end{equation}
and $\{8,12\} <_{F_{1}} \{7,8\}$ implies that
	\begin{equation}\label{eq:12leaves-78}
	w^{2}_{\{7,8\}} \ge w^{1}_{\{7,8\}}.
	\end{equation}

Then Proposition \ref{prop:conedefn} together with (\ref{eq:12leaves-12}), (\ref{eq:12leaves-34}), (\ref{eq:12leaves-56}), and (\ref{eq:12leaves-78}), we have the following chain of inequalities:
	\begin{equation*}\label{eq:cycle}
		w^{1}_{\{1,2\}} \ge w^{2}_{\{1,2\}} > w^{2}_{\{3,4\}} \ge w^{1}_{\{3,4\}} = w^{1}_{\{5,6\}} \ge w^{2}_{\{5,6\}} = w^{2}_{\{7,8\}} \ge w^{1}_{\{7,8\}} = w^{1}_{\{1,2\}}
	\end{equation*}
---a contradiction.  Hence such tree ultrametrics $w, w^{1}, w^{2}$ do not exist.
\end{example}

\subsection{Symmetry on Tropical Line Segments}

In contrast to the previous result, which dealt with how tree topologies vary along tropical line segments, we now turn our focus to understanding when and how {\em invariance} arises in the space of ultrametrics.  To do this, we define the notion of symmetry on ultrametrics in terms of leaf relabeling.  The natural setting for such a study is the action of the symmetric group $\Sym(N)$ on $N$ labels on $[N]$, given by permuting the coordinates (positions) of the labels of the leaves.

In our study, we consider the map $\Sigma: \mathcal{U}_N \times \Sym(N) \to \mathcal{U}_N $ defined by 
\[
\Sigma(w, \sigma) = (w_{\{\sigma_1,\sigma_2\}}, w_{\{\sigma_1,\sigma_3\}}, \ldots , w_{\{\sigma_{N-1},\sigma_{N}\}} ),
\]
where $w$ is an ultrametric and $\sigma \in \Sym(N)$ is an $N$th-order permutation of the symmetric group.

\begin{definition}
\label{def:sym}
Let $T$ be an equidistant tree with $N$ leaves and let $w_T \in \mathcal{U}_N$ be a tree metric associated with $T$.  Define the equivalence relation $\sim_{\sigma}$ between equidistant trees $T$ and $T'$ with $N$ leaves by $T \sim_{\sigma} T'$ if and only if $T$ and $T'$ have the same tree topology and branch lengths, but the labels of leaves in $T'$ are permuted by $\sigma \in \Sym(N)$.
\end{definition}

\begin{example}
\label{ex:perm1}
Let $T_1$ and $T_2$ be the equidistant trees shown in Figure \ref{fig:permutation}.  Here, $T_1 \sim_{\sigma} T_2$, where $\sigma=(2, 3, 1, 4)$.

\begin{figure}[h]
  \centering
    \begin{tikzpicture}[scale=0.75]
    	\draw (0,0) -- (3,5) -- (6,0);
    	\draw (4.8,2) -- (3.6,0);
    	\draw (5.4,1) -- (4.8,0);
    	\filldraw [black] (0,0) circle (1pt);
    	\filldraw [black] (3.6,0) circle (1pt);
    	\filldraw [black] (4.8,0) circle (1pt);
    	\filldraw [black] (6,0) circle (1pt);
    	\node [below] at (0,0) {$4$};
    	\node [below] at (3.6,0) {$3$};
    	\node [below] at (4.8,0) {$2$};
    	\node [below] at (6,0) {$1$};
    	\node at (4.5,3.5) {$0.6$};
    	\node at (5.6,1.6) {$0.2$};
    	\node at (6.25,0.6) {$0.2$};
    	\node at (1,4) {$T_{1}$};
    	
    	\draw (10,0) -- (13,5) -- (16,0);
    	\draw (14.8,2) -- (13.6,0);
    	\draw (15.4,1) -- (14.8,0);
    	\filldraw [black] (10,0) circle (1pt);
    	\filldraw [black] (13.6,0) circle (1pt);
    	\filldraw [black] (14.8,0) circle (1pt);
    	\filldraw [black] (16,0) circle (1pt);
    	\node [below] at (10,0) {$4$};
    	\node [below] at (13.6,0) {$1$};
    	\node [below] at (14.8,0) {$3$};
    	\node [below] at (16,0) {$2$};
    	\node at (14.5,3.5) {$0.6$};
    	\node at (15.6,1.6) {$0.2$};
    	\node at (16.2,0.6) {$0.2$};
    	\node at (11,4) {$T_{2}$};   	
    \end{tikzpicture}
    
  \caption{$T_1$ and $T_2$ in Example \ref{ex:perm1}. }\label{fig:permutation} 
  \end{figure}
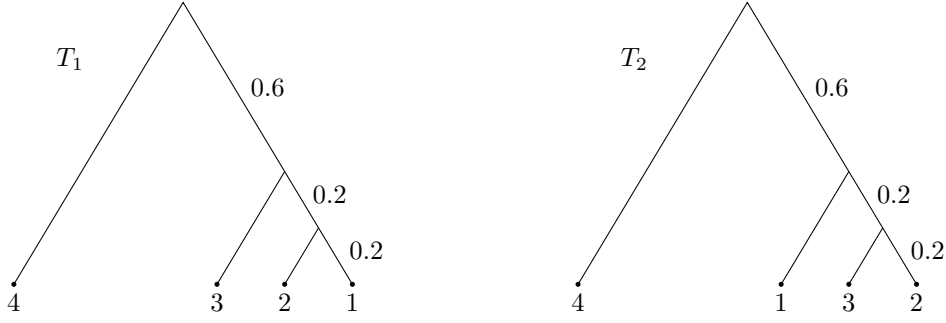

\end{example}

\begin{lemma}
\label{lemma:perm}
Let $T_1, T_2$ be equidistant trees with $N$ leaves.  Let $w^{1}$ be a tree ultrametric for $T_1$ and $w^{2}$ be a tree ultrametric for $T_2$.  Then $T_1 \sim_{\sigma} T_2$ if and only if $\Sigma(w^{1}, \sigma^{-1}) = w^{2}$ or equivalently, $\Sigma(w^{2}, \sigma) = w^{1}$.
\end{lemma}

\begin{proof}
Suppose $T_1 \sim_{\sigma} T_2$.  Let $w^{1} = (w^{1}_{\{1,2\}}, \ldots , w^{1}_{\{N-1,N\}})$ and, similarly, $w^{2} = (w^{2}_{\{1,2\}}, \ldots , w^{2}_{\{N-1,N\}})$.  We have that $w^{1}_{\{i,j\}}$ is the distance from a leaf $i$ to a leaf $j$ on $T_1$, and $w^{2}_{\{i,j\}}$ is the distance from a leaf $i$ to a leaf $j$ on $T_2$, by the definition of a tree metric.  Since $T_1 \sim_{\sigma} T_2$,
we have that $w^{1}_{\{i,j\}} = w^{2}_{\{\sigma_{i}, \sigma_{j}\}}$.  Thus, $\Sigma (w^{1}, \sigma^{-1}) = w^{2}$.

Now suppose $\Sigma (w^1, \sigma^{-1}) = w^{2}$.  Then we have $w^{1}_{\{i,j\}} = w^{2}_{\{\sigma_{i}, \sigma_{j}\}}$ for any pair of leaves $i, j \in [N]$.  Since $w^{1}, \, w^{2} \in \mathcal{U}_N$, there exist equidistant trees $T_1, T_2$ with $N$ leaves, respectively, by Proposition \ref{prop:ultra}.  Then, $w^{1}_{\{i,j\}}$ is the distance from a leaf $i$ to a leaf $j$ on $T_1$, and $w^{2}_{\{i,j\}}$ is the distance from a leaf $i$ to a leaf $j$ on $T_2$, by the definition of a tree metric.  Since $\Sigma (w^{1}, \sigma^{-1}) = w^{2}$, we have that $w^{1}_{\{i,j\}} = w^{2}_{\{\sigma_{i}, \sigma_{j}\}}$ for all pair of leaves $i$ and $j$.  Therefore, $T_1 \sim_{\sigma} T_2$. 
\end{proof}

\begin{example}
\label{ex:perm}
To illustrate Lemma \ref{lemma:perm}, we revisit the equidistant trees $T_1$ and $T_2$ shown in Figure \ref{fig:permutation}.  Their ultrametrics are 
\begin{align*}
w^{1} & = (0.4,\, 0.8,\, 2,\, 0.8,\, 2,\, 2),\\
w^{2} & = (0.8,\, 0.8,\, 2,\, 0.4,\, 2,\, 2).\\
\end{align*}
Then we have $\sigma=(2, 3, 1, 4)$ and $\sigma^{-1} = (3, 1, 2, 4)$, thus
\begin{align*}
\Sigma (w^{2}, \sigma) &= w^{1},\\
\Sigma (w^{1}, \sigma^{-1}) &= w^{2}.\\
\end{align*}
Note that for $T_1:$
\[
\begin{array}{rccccccccc}
w^{1}(1, 2)&=& w^{2}(\sigma_1, \sigma_2)&=& w^{2}(2, 3)&= &&& 0.4\\
w^{1}(1, 3)&=& w^{2}(\sigma_1, \sigma_3)&=& w^{2}(2, 1)&=& w^{2}(1, 2) &=& 0.8\\
w^{1}(1, 4)&=& w^{2}(\sigma_1, \sigma_4)& = & w^{2}(2, 4)&=&&& 2\\
w^{1}(2, 3)&= & w^{2}(\sigma_2, \sigma_3)& = & w^{2}(3, 1)&=& w^{2}(1, 3)& = & 0.8\\
w^{1}(2, 4)&=& w^{2}(\sigma_2, \sigma_4)& = & w^{2}(3, 4)&= &&& 2\\
w^{1}(3, 4)&= & w^{2}(\sigma_3, \sigma_4)& =& w^{2}(1, 4)&=&&& 2\\
\end{array}
\]
and similarly, for $T_2$:
\[
\begin{array}{rccccccccc}
w^{2}(1, 2) &=& w^{1}(\sigma^{-1}_1, \sigma^{-1}_2) &= & w^{1}(3, 1) &= & w^{1}(1, 3) &=& 0.8\\
w^{2}(1, 3) &=& w^{1}(\sigma^{-1}_1, \sigma^{-1}_3) & = & w^{1}(3, 2) &= & w^{1}(2, 3)&= & 0.8\\
w^{2}(1, 4) &=& w^{1}(\sigma^{-1}_1, \sigma^{-1}_4) & = & w^{1}(3, 4) &= &&&2\\
w^{2}(2, 3) &= & w^{1}(\sigma^{-1}_2, \sigma^{-1}_3) & = & w^{1}(1, 2) &= &&&0.4\\
w^{2}(2, 4) &=& w^{1}(\sigma^{-1}_2, \sigma^{-1}_4) & = & w^{1}(1, 4) &= &&&2\\
w^{2}(3, 4) &=& w^{1}(\sigma^{-1}_3, \sigma^{-1}_4) & = & w^{1}(2, 4) &=&&&2\\
\end{array}
\]
\end{example}

\smallskip

With this definition of symmetry given by permutation of leaf relabeling, we now study how symmetry behaves on tropical line segments.  Let $\Gamma (w^{T})$ be a tropical line segment from the origin $(0, 0,\ldots , 0)$ to the ultrametric $w^{T} \in \mathcal{U}_N$ associated with an equidistant tree $T$.  Also, let $\Gamma (w^{T}, w^{T_0})$ be a tropical line segment from the ultrametric $w^{T_0} \in \mathcal{U}_N$  associated with an equidistant tree $T_0$, to the ultrametric $w^{T} \in \mathcal{U}_N$ associated with an equidistant tree $T$.  These definitions give the following result.

\begin{proposition}\label{prop:perm1}
Suppose $T_1$ is an equidistant tree with $N$ leaves and let $T_2$ be an equidistant tree such that $T_1 \sim_{\sigma} T_2$.  Then
$$
\Sigma(\Gamma(w^{1}), \sigma) = \Gamma (w^{2})
$$
and
$$
\Sigma(\Gamma(w^{2}), \sigma^{-1}) = \Gamma (w^{1}).
$$
\end{proposition}

\begin{proof}
Let $w^{1} =(w^{1}_{\{1,2\}}, w^{1}_{\{1,3\}}, \ldots , w^{1}_{\{N-1,N\}})$ be an ultrametric computed from an equidistant tree $T_1$, and similarly, $w^{2}=(w^{2}_{\{1,2\}}, w^{2}_{\{1,3\}}, \ldots , w^{2}_{\{N-1,N\}})$ be an ultrametric computed from an equidistant tree $T_2$.  Order the coordinates of both $w^{1}$ and $w^{2}$ from the smallest and largest, and let $(w^{1}_{(1)}, \ldots , w^{1}_{(n)})$ and $(w^{2}_{(1)}, \ldots , w^{2}_{(n)})$ be the ultrametrics after ordering the coordinates of $w^{1}$ and $w^{2}$, respectively, from the smallest to the largest coordinates.  By Lemma \ref{lemma:perm}, $\Sigma(w^{1}, \sigma^{-1}) = w^{2}$ since $T_1 =_{\sigma} T_2$.  Thus, $w^{1}_{(i)} = w^{2}_{(i)}$ for $i = 1, \ldots , {n}$.  By applying the algorithm in the proof of Proposition 5.2.5 in \cite{maclagan2015introduction} on the line segment from the origin $(0, 0, \ldots , 0)$ to $(w^{1}_{(1)}, \ldots , w^{1}_{(n)})$, we obtain a tropical line segment from the origin to $w^{1}$.  Since $w^{1}_{(i)} = w^{2}_{(i)}$ for $i = 1, \ldots , {n}$, we have that $\Sigma(\Gamma(w^{1}), \sigma) = \Gamma (w^{2})$ or equivalently $\Sigma(\Gamma(w^{2}), \sigma^{-1}) = \Gamma (w^{1})$.
\end{proof}

The following results provide a formalization of invariance under the action of permutation of leaf labels in terms of tree topologies.  If two tree ultrametrics have the same topology and branch lengths, but differ by a permutation of leaf labels, and, correspondingly, two other tree ultrametrics have the same properties, and the orbit of the symmetric group action permuting the labels is the same for both sets of leaf-permuted trees, then the tropical line segments connecting these pairs coincide.

\begin{theorem}\label{thm:perm2}
Suppose $T_0$ and $T'_0$ are equidistant trees with $N$ leaves such that  $T_0 =_{\sigma} T'_0$.  Also, suppose $T$ and $T'$ are equidistant trees with $N$ leaves such that $T =_{\sigma} T'$.  Then
$$
\Sigma(\Gamma(w^{T}, w^{T_0}), \sigma) = \Gamma (w^{T'}, w^{T'_0})
$$
and
$$
\Sigma(\Gamma(w^{T'}, w^{T'_0}), \sigma^{-1}) = \Gamma (w^{T}, w^{T_0}).
$$
\end{theorem}

\begin{proof}
Since $T_0 \sim_{\sigma} T'_0$, we have that the differences $w^{T} - w^{T_0}$ and $w^{T'} - w^{T'_0}$ are equal after ordering the coordinates of $w^{T} - w^{T_0}$ and $w^{T'} - w^{T'_0}$ from the smallest and largest.  The remainder of the proof follows the proof of Proposition \ref{prop:perm1}.
\end{proof}

\begin{remark}
Haar measures are a natural generalization of Lebesgue measures on spaces with a specified group structure; they are relevant to fundamental studies in probability theory.  By our specification in Definition \ref{def:sym} and Theorem \ref{thm:perm2}, and the existence of probability measures on the tropical geometric interpretation of tree space \citep{monod2018tropical}, there exist Haar measures: the group structure is given by the symmetric group $\mathrm{Sym}(N)$ and the invariance is on tropical line segments between ultrametrics.
\end{remark}

\begin{corollary}
For $T_0$ invariant under $\sigma$, let $T$ and $T'$ be equidistant trees with $N$ leaves such that $T \sim_{\sigma} T'$.  Then
$$
\Sigma(\Gamma(w^{T}, w^{T_0}), \sigma) = \Gamma (w^{T'}, w^{T_0})
$$
and
$$
\Sigma(\Gamma(w^{T'}, w^{T_0}), \sigma^{-1}) = \Gamma (w^{T}, w^{T_0}).
$$
\end{corollary}

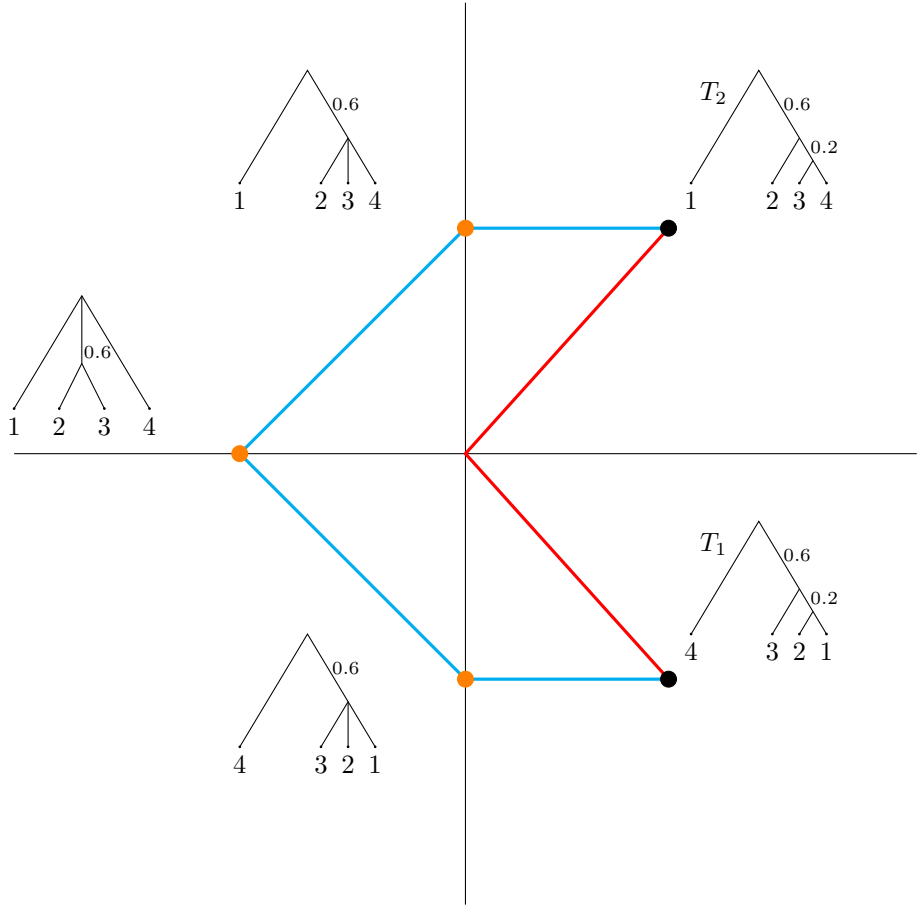
\begin{figure}[h]
  \begin{center}
     \begin{tikzpicture}[scale=0.3]
     	% the coordinate axis
     	\draw (-20,0) -- (20,0);
     	\draw (0,-20) -- (0,20);
     	
     	% the blue and red lines     	
     	\draw [red,very thick] (9,10) -- (0,0) -- (9,-10);
     	\draw [cyan,very thick] (9,10) -- (0,10) -- (-10,0) -- (0,-10) -- (9,-10); 
     	
     	% nodes     	
     	\filldraw [black] (9,10) circle (10pt);
     	\filldraw [black] (9,-10) circle (10pt);
     	\filldraw [orange] (0,10) circle (10pt);
     	\filldraw [orange] (-10,0) circle (10pt);
     	\filldraw [orange] (0,-10) circle (10pt);
     	
     	% T0     	
     	\draw (-20,2) -- (-17,7) -- (-14,2);
     	\draw (-17,7) -- (-17,4) -- (-18,2);
     	\draw (-17,4) -- (-16,2);
     	\filldraw [black] (-20,2) circle (1pt);
     	\filldraw [black] (-18,2) circle (1pt);
     	\filldraw [black] (-16,2) circle (1pt);
     	\filldraw [black] (-14,2) circle (1pt);
     	\node [below] at (-20,2) {$1$};
     	\node [below] at (-18,2) {$2$};
     	\node [below] at (-16,2) {$3$};
     	\node [below] at (-14,2) {$4$};
     	\node at (-16.3,4.5) {\scriptsize $0.6$};
%     	\node at (-19,6) {$T_{0}$};
     	
     	% T1
     	\draw (10,-8) -- (13,-3) -- (16,-8);
     	\draw (14.8,-6) -- (13.6,-8);
     	\draw (15.4,-7) -- (14.8,-8);
     	\filldraw [black] (10,-8) circle (1pt);
     	\filldraw [black] (13.6,-8) circle (1pt);
     	\filldraw [black] (14.8,-8) circle (1pt);
     	\filldraw [black] (16,-8) circle (1pt);
     	\node [below] at (10,-8) {$4$};
     	\node [below] at (13.6,-8) {$3$};
     	\node [below] at (14.8,-8) {$2$};
     	\node [below] at (16,-8) {$1$};
     	\node at (14.7,-4.5) {\scriptsize $0.6$};
     	\node at (15.9,-6.4) {\scriptsize $0.2$};
     	\node at (11,-4) {$T_{1}$};
     	
     	% T2
     	\draw (10,12) -- (13,17) -- (16,12);
     	\draw (14.8,14) -- (13.6,12);
     	\draw (15.4,13) -- (14.8,12);
     	\filldraw [black] (10,12) circle (1pt);
     	\filldraw [black] (13.6,12) circle (1pt);
     	\filldraw [black] (14.8,12) circle (1pt);
     	\filldraw [black] (16,12) circle (1pt);
     	\node [below] at (10,12) {$1$};
     	\node [below] at (13.6,12) {$2$};
     	\node [below] at (14.8,12) {$3$};
     	\node [below] at (16,12) {$4$};
     	\node at (14.7,15.5) {\scriptsize $0.6$};
     	\node at (15.9,13.6) {\scriptsize $0.2$};
     	\node at (11,16) {$T_{2}$};
     	
     	% other trees     	
     	\draw (-10,12) -- (-7,17) -- (-4,12);
     	\draw (-5.2,14) -- (-6.4,12);
     	\draw (-5.2,14) -- (-5.2,12);
     	\filldraw [black] (-10,12) circle (1pt);
     	\filldraw [black] (-6.4,12) circle (1pt);
     	\filldraw [black] (-5.2,12) circle (1pt);
     	\filldraw [black] (-4,12) circle (1pt);
     	\node [below] at (-10,12) {$1$};
     	\node [below] at (-6.4,12) {$2$};
     	\node [below] at (-5.2,12) {$3$};
     	\node [below] at (-4,12) {$4$};
     	\node at (-5.3,15.5) {\scriptsize $0.6$};
     	
     	\draw (-10,-13) -- (-7,-8) -- (-4,-13);
     	\draw (-5.2,-11) -- (-6.4,-13);
     	\draw (-5.2,-11) -- (-5.2,-13);
     	\filldraw [black] (-10,-13) circle (1pt);
     	\filldraw [black] (-6.4,-13) circle (1pt);
     	\filldraw [black] (-5.2,-13) circle (1pt);
     	\filldraw [black] (-4,-13) circle (1pt);
     	\node [below] at (-10,-13) {$4$};
     	\node [below] at (-6.4,-13) {$3$};
     	\node [below] at (-5.2,-13) {$2$};
     	\node [below] at (-4,-13) {$1$};
     	\node at (-5.3,-9.5) {\scriptsize $0.6$};
     	
     \end{tikzpicture}
  \end{center}
  \caption{Equidistant trees $T_1$ and $T_2$, with tropical line segment (blue line) and the BHV geodesic (red line) for Examples \ref{ex:perm2} and \ref{ex:algorithm}.}\label{fig:permutation2} 
  \end{figure}

\begin{example}\label{ex:perm2}
Let $T_1, T_2$ be equidistant trees in Figure \ref{fig:permutation2}.  Here, $\sigma=\sigma^{-1}=(4,3,2,1)$.  Notice that $\Sigma(\Gamma(w^{1}, w^{0}), \sigma) = \Gamma (w^{2}, w^{0})$ and $\Sigma(\Gamma(w^{2}, w^{0}), \sigma^{-1}) = \Gamma(w^{1}, w^{0})$.

In this figure, the black points represent the trees $T_1$ and $T_2$.  These trees have the same tree topology and the same branch lengths, but leaves are labeled differently.  In BHV space, $T_1$ and $T_2$ are distinct trees, and thus belong to different orthants in this example, since there are only two internal edges in these trees: recall that in BHV space, within each orthant, trees are stored by their internal edge lengths, which represent coordinates.  The BHV distance between $T_1$ and $T_2$ is then the unique geodesic that is the cone path, traversing the origin, illustrated by the red line.  (See Figure \ref{fig:petersen} for the configuration of $\mathrm{BHV}_5$ and to see that $T_1$ and $T_2$ are not in neighboring orthants.)

We have $T_1 \sim_{\sigma} T_2$, and the tropical line segment is illustrated by the blue line.  See the details in Example \ref{ex:algorithm} for an explicit computation of the tropical line segment between these trees.  The orange points represent the three remaining trees in the figure: notice that in these trees, there is only one internal edge length, and thus the orange points lie on the $1$-dimensional strata partitioning the quadrants.  The line segments connecting the points traverse the quadrants, and at every point along the blue lines, there are two internal edge lengths until the subsequent orange point is reached, where one internal edge contracts completely into one internal node.
\end{example}

%Notice in this example that the bending points of the tropical line segment (the orange points) exhibit interesting tree topologies, including one (2, 2, 2, 0.8, 2, 2) that exhibits sister taxa.  Recall from Section \ref{description:BHV} above that orthant boundaries (also referred to as {\em arrays}) in BHV space that boundaries correspond to trees with a collapsed internal edge between two orthants with similar tree topologies, however, there exist trees in BHV space with no specific array assigned to their particular topology.  The sister taxa in this example corresponds to a polytomy, which ordinarily lie on array, though this tree topology 

Notice in this example that the bending points of the tropical line segment (the orange points) exhibit interesting tree topologies, including one (2, 2, 2, 0.8, 2, 2) that exhibits sister taxa.  Recall from Section \ref{description:BHV} above that orthant boundaries in BHV space that boundaries correspond to trees with a collapsed internal edge between two orthants with similar tree topologies; this tree then lies on a BHV orthant boundary.  Depending on conventions adopted, however, it may not always be a tree topology; if tree topologies are required to be full dimensional, then this tree is not a valid tree topology in BHV space.  

\subsection{Computing Tropical Line Segments}
\label{subsec:tls}

We now focus our study on the tropical line segment and give an algorithm for its computation.  We begin with a remark from Example \ref{ex:perm2}, where we see that the tropical line segment is not a cone path; it does not traverse the origin (or star tree), whereas BHV geodesics are often cone paths when orthants are far apart.  This observation leads to the following proposition, which uncovers a more subtle behavior of tropical line segments compared to BHV geodesics.

\begin{proposition}
\label{prop:noconepath}
For two points $u$ and $v$ in general position in the tropical projective torus $\R^{n}/\R\one$, the tropical line segment between $u$ and $v$ does not contain the origin.
\end{proposition}
 
\begin{proof}[Proof due to Carlos Am\'endola]
    Let $u=(u_1,\dots,u_{n})$ and $v=(v_1,\dots,v_{n})$ be two points in general position in $\R^{n}/\R\one$.  The tropical line segment consists of the concatenation of $m$ ordinary line segments, each one having a direction of a zero-one vector \cite[Proposition 5.2.5]{maclagan2015introduction}. The bending points (i.e., the points at which the segments are concatenated) can be computed explicitly via the entries of the difference vector $\lambda = v - u$. Up to reordering, they are given by $$\lambda_i \odot u \boxplus v $$ for $i=1,\dots,n$ (which include the endpoints $u$ and $v$). 
    From these expressions we see that in order for $0$ to be contained in one of the ordinary line segments that comprise the tropical segment from $u$ to $v$, the bending points would need to be in a particular arrangement (one must be a scalar multiple of the other), which does not happen because $u$ and $v$ are in general position.
\end{proof}

Notice that Proposition \ref{prop:noconepath} also follows from the fact that there is a unique tropical line segment passing through two points in general position: We may assume without loss of generality that one of the points is 0 and set the other to be $u$, let $L$ be the unique tropical line segment passing through 0 and $u$.  Then for any $v \not\in L$, there is no tropical line segment containing $u,$ $v,$ and 0.  Hence, the tropical line segment between $u$ and $v$ cannot contain 0.

We now give an algorithm for computing tropical line segments between equidistant trees with $N$ leaves.  The algorithm takes as input two ultrametrics: $w^1 = \big(w^1_{\{1, 2\}}, \ldots, w^1_{\{N-1, N\}} \big)$ associated with an equidistant tree $T_1$ with $N$ leaves; and $w^2 = \big(w^2_{\{1, 2\}}, \ldots, w^2_{\{N-1, N\}} \big)$ associated with an equidistant tree $T_2$ with $N$ leaves.  It returns the tropical line segment between $T_1$ and $T_2$ in $\mathcal{U}_N$.

\begin{algorithm}[h]
\caption{Algorithm for Computing Tropical Line Segments Between Ultrametrics}\label{algorithm:troplineseg}
\begin{algorithmic}[1]
\Procedure{TropLine}{$w^1,w^2$}\Comment{Tropical line segment between ultrametrics $w^1$ and $w^2$}
\State $w^1\gets \big(w^1_{\{1, 2\}},\, \ldots,\, w^1_{\{N-1, N\}} \big)$
\State $w^2 \gets \big(w^2_{\{1, 2\}},\, \ldots,\, w^2_{\{N-1, N\}} \big)$
\State $\lambda' \gets \big( w^2_{\{1, 2\}} - w^1_{\{1, 2\}}, \, \ldots, \, w^2_{\{N-1, N\}} - w^1_{\{N-1, N\}} \big)$\Comment{Compute the difference between $T_2$ and $T_1$}
\State $\lambda \gets \big( \lambda'_{(1)}, \, \ldots, \, \lambda'_{(n)} \big)$\Comment{Reorder the elements of $\lambda'$ from smallest to largest}
\State $L = \emptyset$
\For{$i = 1, \ldots, n$}
\State $w_{\{i, i+1\}} \gets \lambda_{i} \odot w^1_{\{i, i+1\}}$
\State $y_{\{i, i+1\}} \gets w_{\{i,i+1\}} \boxplus w^2_{\{i, i+1\}}$
\EndFor
\State $y \gets \big( y_{\{1,2\}},\, \ldots,\, y_{\{N-1,N\}} \big)$
\If{$\exists\ y_i > 2$}
\State $\mathrm{increment} \gets \max(y_i) - 2$
\For{$i = 1, \ldots, n$}
\State $y_i \gets y_i - \mathrm{increment}$\Comment{Rescale branch lengths to obtain a unitary equidistant tree}
\State $y \gets \big( y_{1},\, \ldots,\, y_{n} \big)$
\EndFor
\Else 
\State $y \gets \big( y_{\{1,2\}},\, \ldots,\, y_{\{N-1,N\}} \big)$
\EndIf
\State $L \gets L \cup \{y\}$
\State \textbf{return} $L$\Comment{Tropical line segments connecting the ultrametrics in $L$}
\EndProcedure
\end{algorithmic}
\end{algorithm}

We remark here that there is a similarity in approach between our algorithm and various results by \cite{bernstein2020infinity} that combinatorially computes ultrametrics.  Theorem 3.6 in \cite{bernstein2020infinity} gives a combinatorial description of a finite set of ultrametrics with a tropical convex hull that itself is a set of of ultrametrics; it begins with an ultrametric that is coordinate-wise larger than a vertex of a particular tropical polytope and also ``slides'' internal nodes down until another candidate vertex is attained, which, however, may not necessarily be a tropical vertex \citep{yu2020extreme}.  In our algorithm, various tree topologies are obtained at intermediate steps also by ``sliding'' internal nodes.

%Theorem 3.6 gives a combinatorial description of a finite set of ultrametrics
%whose tropical convex hull is the set of ultrametrics nearest in the l\infty -norm to a given
%dissimilarity map.

The following example illustrates the implementation of Algorithm \ref{algorithm:troplineseg}.

\begin{example}\label{ex:algorithm}
We revisit the equidistant trees $T_2$ and $T_1$ from Figure \ref{fig:permutation2}, and compute the tropical line segment (blue line) in the figure using Algorithm \ref{algorithm:troplineseg}, from $T_2$ to $T_1$.  Their ultrametrics are
\begin{align*}
w^{2} & = (2,\, 2,\, 2,\, 0.8,\, 0.8,\, 0.4),\\
w^{1} & = (0.4,\, 0.8,\, 2,\, 0.8,\, 2,\, 2).\\
\end{align*}
First, we calculate the difference between $w^2$ and $w^1$, $\lambda' = (1.6,\, 1.2,\, 0,\, 0,\, -1.2,\, -1.6)$, and reorder the elements of $\lambda'$ from smallest to largest: $\lambda = (-1.6,\, -1.2,\, 0,\, 1.2,\, 1.6).$

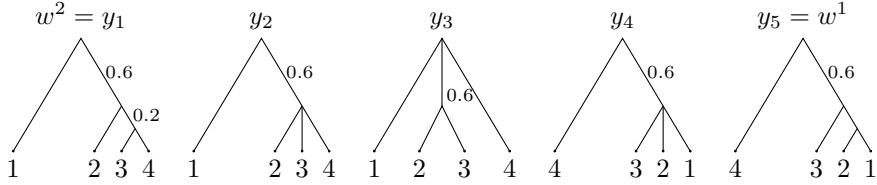
\begin{figure}[h]
\begin{center}
\begin{tikzpicture}[scale=0.3]
     	\draw (10,12) -- (13,17) -- (16,12);
     	\draw (14.8,14) -- (13.6,12);
     	\draw (15.4,13) -- (14.8,12);
     	\filldraw [black] (10,12) circle (1pt);
     	\filldraw [black] (13.6,12) circle (1pt);
     	\filldraw [black] (14.8,12) circle (1pt);
     	\filldraw [black] (16,12) circle (1pt);
     	\node [below] at (10,12) {$1$};
     	\node [below] at (13.6,12) {$2$};
     	\node [below] at (14.8,12) {$3$};
     	\node [below] at (16,12) {$4$};
     	\node at (14.7,15.5) {\scriptsize $0.6$};
     	\node at (15.9,13.6) {\scriptsize $0.2$};
	\node [above] at (13,17) {$w^2 = y_{1}$};

     	\draw (18,12) -- (21,17) -- (24,12);
     	\draw (22.8,14) -- (21.6,12);
	\draw (22.8,12) -- (22.8, 14);
     	\filldraw [black] (18,12) circle (1pt);
     	\filldraw [black] (24,12) circle (1pt);
     	\filldraw [black] (22.8,12) circle (1pt);
     	\filldraw [black] (21.6,12) circle (1pt);
	\node [below] at (18,12) {$1$};
     	\node [below] at (24,12) {$4$};
     	\node [below] at (22.8,12) {$3$};
     	\node [below] at (21.6,12) {$2$};
	\node at (22.7,15.5) {\scriptsize $0.6$};
	\node [above] at (21,17) {$y_{2}$};

     	\draw (26,12) -- (29,17) -- (32,12);
     	\draw (29,17) -- (29,14);
	\draw (28,12) -- (29,14) -- (30,12);
	\filldraw [black] (26,12) circle (1pt);
     	\filldraw [black] (32,12) circle (1pt);
     	\filldraw [black] (28,12) circle (1pt);
     	\filldraw [black] (30,12) circle (1pt);
	\node [below] at (26,12) {$1$};
     	\node [below] at (32,12) {$4$};
     	\node [below] at (28,12) {$2$};
     	\node [below] at (30,12) {$3$};
	\node at (29.8,14.5) {\scriptsize $0.6$};
	\node [above] at (29,17) {$y_{3}$};
	
     	\draw (34,12) -- (37,17) -- (40,12);
     	\draw (38.8,14) -- (37.6,12);
	\draw (38.8,12) -- (38.8,14);
     	\filldraw [black] (34,12) circle (1pt);
     	\filldraw [black] (40,12) circle (1pt);
     	\filldraw [black] (37.6,12) circle (1pt);
     	\filldraw [black] (38.8,12) circle (1pt);
	\node [below] at (34,12) {$4$};
     	\node [below] at (40,12) {$1$};
     	\node [below] at (37.6,12) {$3$};
     	\node [below] at (38.8,12) {$2$};
	\node at (38.7,15.5) {\scriptsize $0.6$};
	\node [above] at (37,17) {$y_{4}$};

     	\draw (42,12) -- (45,17) -- (48,12);
     	\draw (46.8,14) -- (45.6,12);
     	\draw (47.4,13) -- (46.8,12);
     	\filldraw [black] (42,12) circle (1pt);
     	\filldraw [black] (48,12) circle (1pt);
     	\filldraw [black] (45.6,12) circle (1pt);
     	\filldraw [black] (46.8,12) circle (1pt);
	\node [below] at (42,12) {$4$};
     	\node [below] at (48,12) {$1$};
     	\node [below] at (45.6,12) {$3$};
     	\node [below] at (46.8,12) {$2$};
	\node at (46.7,15.5) {\scriptsize $0.6$};
	\node [above] at (45,17) {$y_{5} = w^1$};
	
\end{tikzpicture}
  \caption{Trees $y_i$ for $i=1,\ldots,5$ given by Algorithm \ref{algorithm:troplineseg} from $T_2$ to $T_1$ given in Figure \ref{fig:permutation2}.}\label{fig:exalgorithm2} 
\end{center}
\end{figure}

\begin{enumerate}[1.]
\item For $\lambda_{1} = -1.6$, we have
\begin{align*}
w_1 := \lambda_{1} \odot w^1&= (-1.6+0.4,\, -1.6+0.8,\, -1.6+2,\, -1.6+0.8,\, -1.6+2,\, -1.6+2)\\
&= (-1.2,\, -0.8,\, 0.4,\, -0.8,\, 0.4,\, 0.4)\\
y_1 := w_1 \boxplus w^2 &= \big( \max(-1.2, 2),\, \max(-0.8, 2),\, \max(0.4, 2),\, \max(-0.8, 0.8),\, \max(0.4, 0.8),\, \max(0.4, 0.4) \big)\\
&= w^2 = (2,\, 2,\, 2,\, 0.8,\, 0.8,\, 0.4) 
\end{align*}

\item For $\lambda_{2} = -1.2$, we have
\begin{align*}
w_2 := \lambda_{2} \odot w^1&= (-1.2+0.4,\, -1.2+0.8,\, -1.2+2,\, -1.2+0.8,\, -1.2+2,\, -1.2+2)\\
&= (-0.8,\, -0.4,\, 0.8,\, -0.4,\, 0.8,\, 0.8)\\
y_2 := w_2 \boxplus w^2 &= \big( \max(-0.8, 2),\, \max(-0.4, 2),\, \max(0.8, 2),\, \max(-0.4, 0.8),\, \max(0.8, 0.8),\, \max(0.8, 0.4) \big)\\
&= (2,\, 2,\, 2,\, 0.8,\, 0.8,\, 0.8)
\end{align*}

\item For $\lambda_{3} = 0$, we have
\begin{align*}
 w_3 := \lambda_{3} \odot w^1 = w^1 &= (0.4,\, 0.8,\, 2,\, 0.8,\, 2,\, 2)\\
%&= (0.4,\, 0.8,\, 2,\, 0.8,\, 2,\, 2)\\
y_3 := w_3 \boxplus w^2 &= \big( \max(0.4, 2),\, \max(0.8, 2),\, \max(2, 2),\, \max(0.8, 0.8),\, \max(2, 0.8),\, \max(2, 0.4) \big)\\
&= (2,\, 2,\, 2,\, 0.8,\, 2,\, 2)
\end{align*}

\item For $\lambda_{4} = 1.2$, we have
\begin{align*}
w_4 := \lambda_{4} \odot w^1&= (1.2+0.4,\, 1.2+0.8,\, 1.2+2,\, 1.2+0.8,\, 1.2+2,\, 1.2+2)\\
&= (1.6,\, 2,\, 3.2,\, 2,\, 3.2,\, 3.2)\\
y_4 := w_4 \boxplus w^2 &= \big( \max(1.6, 2),\, \max(2, 2),\, \max(3.2, 2),\, \max(2, 0.8),\, \max(3.2, 0.8),\, \max(3.2, 0.4) \big)\\
&= (2,\, 2,\, 3.2,\, 2,\, 3.2,\, 3.2)\\
& = (0.8,\, 0.8,\, 2,\, 0.8,\, 2,\, 2)
\end{align*}

\item For $\lambda_{5} = 1.6$, we have
\begin{align*}
w_5 := \lambda_{5} \odot w^1&= (1.6+0.4,\, 1.6+0.8,\, 1.6+2,\, 1.6+0.8,\, 1.6+2,\, 1.6+2)\\
&= (2,\, 2.4,\, 3.6,\, 2.4,\, 3.6,\, 3.6)\\
y_5 := w_5 \boxplus w^2 &= \big( \max(2, 2),\, \max(2.4, 2),\, \max(3.6, 2),\, \max(2.4, 0.8),\, \max(3.6, 0.8),\, \max(3.6, 0.4) \big)\\
&= (2,\, 2.4,\, 3.6,\, 2.4,\, 3.6,\, 3.6)\\
&= (0.4,\, 0.8,\, 2,\, 0.8,\, 2,\, 2) = w^1
\end{align*}
\end{enumerate}

The trees $y_i$ that we obtain for each value of $\lambda_i$ are shown in sequence in Figure \ref{fig:exalgorithm2}.  Notice that $w^2 = y_1 =_\sigma y_5 = w^1$, and $y_2 =_\sigma y_4$.

\end{example}

Currently, the fastest available algorithm to compute BHV geodesics runs in quartic time, $O(N^4)$ \cite{Owen:2011:FAC:1916480.1916603}.

\begin{proposition}{\cite[Proposition 5.2.5]{maclagan2015introduction}}
The time complexity to compute the tropical line segment connecting two points in $\mathbb{R}^n/\mathbb{R}\one$ is $O(n\, \log\, n) = O(N^2\, \log\, N)$.
\end{proposition}

Notice that the computational complexity of the tropical line segment corresponds to the dimension of the tropical projective torus in the sense that the number of bending points is linear with respect to the dimension of the tropical projective torus.  Computing the BHV metric, on the other hand, depends on the angles formed between paths that pass through points lying in different orthants; these paths connect to form the geodesic in BHV space.  In this sense, computing the BHV metric also computes line segments but the computation is based on compatibility of two tree topologies while the computation of the tropical line segment does not depend on the tree topology and only computes line segments based on the coordinates of the ultrametrics.

%For the tropical this is simply the dimension of the tropical projective torus.  The number of bending points are linear in terms of the dimension of the tropical projective torus (or space).  For BHV metric, this depends on the computing the angles of the points through different orthants and the number of orthants between these points (see Figure 2 in this example https://www.researchgate.net/publication/319035717_A_combinatorial_method_for_connecting_BHV_spaces_representing_different_numbers_of_taxa/figures?lo=1 ).  Basically computing the geodesic in the BHV metric also computes the line segments but it based on the compatibilities of two tree topologies but the tropical line segment only compute the line segments based on the coordinates of ultrametrics.  

%%%%%%%%%%%%%%%%%%%%%%%%%%%%%%%%%%%%%%%%%%%%%%%%%%%

\section{Discussion}
\label{sec:end}

In this work, we considered the space of phylogenetic trees in the context of tropical geometry as an alternative framework to BHV space.  The construction of BHV space is based on tree topologies, where a Euclidean orthant is assigned to each tree topology.  In our paper, we study tree topologies and their occurrence in the tropical geometric phylogenetic tree space.  In particular, we use the tropical line segment as a framework for our study.  For any two given trees, we compute the tropical line segment between them and prove a combinatorial theorem that describes all tree topologies that occur on this tropical line segment between the two trees.  We also provide a notion of invariance on a tropical line segment, based on a permutation of leaf labels on trees; this construction has implications in tropical geometric applications to probability theory, since it provides the existence of Haar measures on phylogenetic tree space.  We also give an algorithm to compute the tropical line segment between any two trees, which has a lower computational complexity than the current state-of-the-art for computing geodesics in BHV space.  We also show that the tropical line segment does not pass through the origin, whereas in BHV space, for orthants that are far apart from one another, the geodesic is often a cone path.  This implies a more subtle and intricate geometry than that of BHV space, which is an interesting avenue for further study.

Our work lays foundations for future studies in both theoretical and applied directions.  The behavior of the tropical line segment inspires further questions concerning the geometry of tropical geometric phylogenetic tree space.  One example is the study of the curvature of the space; given that general geodesics are not unique in palm tree space (i.e., the tropical geometric tree space endowed with the tropical metric), it differs from the $\mathrm{CAT}(0)$ geometry of BHV space.  The variation of tree topologies in tropical geometric phylogenetic tree space may also be used for statistical studies.  For example, when different tree topologies arise via a random data generating processes, an interesting question is to ask whether the the difference in topologies is due to a difference in distributions.  Palm tree space is a well-defined probability space \citep{monod2018tropical}, so this question may be posed in terms of a statistical hypothesis test.  A natural way to measure differences between two objects is via a metric; \cite{10.1093/sysbio/syx046} propose a metric on phylogenetic tree shapes, which may be used to define a test statistic.

%%%%%%%%%%%%%%%%%%%%%%%%%%%%%%%%%%%%%%%%%%%%%%%%%%%

\section*{Acknowledgments}
We wish to thank Carlos Am\'{e}ndola for many helpful discussions and his insightful input and for the proof of Proposition \ref{prop:noconepath}.

R.Y.~is supported in part by NSF DMS \#1622369 and \#1916037.  Any opinions, findings, and conclusions or recommendations expressed in this material are those of the author(s) and do not necessarily reflect the views of any of the funders.

\clearpage
\newpage
\bibliographystyle{chicago}  % or choose another bib list style
\bibliography{TreeTop_ref} % edit bibexample.bib file ...

\begin{thebibliography}{}

\bibitem[\protect\citeauthoryear{Akian, Gaubert, Ni\c{t}ic\u{a}, and
  Singer}{Akian et~al.}{2011}]{AKIAN20113261}
Akian, M., S.~Gaubert, V.~Ni\c{t}ic\u{a}, and I.~Singer (2011).
\newblock Best {A}pproximation in {M}ax-plus {S}emimodules.
\newblock {\em Linear Algebra and its Applications\/}~{\em 435\/}(12),
  3261--3296.

\bibitem[\protect\citeauthoryear{Alberich, Cardona, Rosselló, and
  Valiente}{Alberich et~al.}{2009}]{ALBERICH20091320}
Alberich, R., G.~Cardona, F.~Rosselló, and G.~Valiente (2009).
\newblock An algebraic metric for phylogenetic trees.
\newblock {\em Applied Mathematics Letters\/}~{\em 22\/}(9), 1320--1324.

\bibitem[\protect\citeauthoryear{Allen and Steel}{Allen and
  Steel}{2001}]{Allen2001}
Allen, B.~L. and M.~Steel (2001).
\newblock Subtree {T}ransfer {O}perations and {T}heir {I}nduced {M}etrics on
  {E}volutionary {T}rees.
\newblock {\em Annals of Combinatorics\/}~{\em 5\/}(1), 1--15.

\bibitem[\protect\citeauthoryear{Allman and Rhodes}{Allman and
  Rhodes}{2008}]{allman2008phylogenetic}
Allman, E.~S. and J.~A. Rhodes (2008).
\newblock Phylogenetic ideals and varieties for the general {M}arkov model.
\newblock {\em Advances in Applied Mathematics\/}~{\em 40\/}(2), 127--148.

\bibitem[\protect\citeauthoryear{Ardila}{Ardila}{2005}]{ardila2005subdominant}
Ardila, F. (2005).
\newblock Subdominant matroid ultrametrics.
\newblock {\em Annals of Combinatorics\/}~{\em 8\/}(4), 379--389.

\bibitem[\protect\citeauthoryear{Ardila and Klivans}{Ardila and
  Klivans}{2006}]{ARDILA200638}
Ardila, F. and C.~J. Klivans (2006).
\newblock The {B}ergman {C}omplex of a {M}atroid and {P}hylogenetic {T}rees.
\newblock {\em Journal of Combinatorial Theory, Series B\/}~{\em 96\/}(1),
  38--49.

\bibitem[\protect\citeauthoryear{Bandelt and Dress}{Bandelt and
  Dress}{1992}]{BANDELT199247}
Bandelt, H.-J. and A.~W. Dress (1992).
\newblock A canonical decomposition theory for metrics on a finite set.
\newblock {\em Advances in Mathematics\/}~{\em 92\/}(1), 47--105.

\bibitem[\protect\citeauthoryear{Bernstein}{Bernstein}{2020}]{bernstein2020infinity}
Bernstein, D.~I. (2020).
\newblock L-infinity optimization to {B}ergman fans of matroids with an
  application to phylogenetics.
\newblock {\em SIAM Journal on Discrete Mathematics\/}~{\em 34\/}(1), 701--720.

\bibitem[\protect\citeauthoryear{Bernstein and Long}{Bernstein and
  Long}{2017}]{bernstein2017infinity}
Bernstein, D.~I. and C.~Long (2017).
\newblock L-infinity optimization to linear spaces and phylogenetic trees.
\newblock {\em SIAM Journal on Discrete Mathematics\/}~{\em 31\/}(2), 875--889.

\bibitem[\protect\citeauthoryear{Billera, Holmes, and Vogtmann}{Billera
  et~al.}{2001}]{BILLERA2001733}
Billera, L.~J., S.~P. Holmes, and K.~Vogtmann (2001).
\newblock Geometry of the {S}pace of {P}hylogenetic {T}rees.
\newblock {\em Advances in Applied Mathematics\/}~{\em 27\/}(4), 733--767.

\bibitem[\protect\citeauthoryear{Buneman}{Buneman}{1974}]{BUNEMAN197448}
Buneman, P. (1974).
\newblock A {N}ote on the {M}etric {P}roperties of {T}rees.
\newblock {\em Journal of Combinatorial Theory, Series B\/}~{\em 17\/}(1),
  48--50.

\bibitem[\protect\citeauthoryear{Cardona, Mir, Rossell{\'o}, Rotger, and
  S{\'a}nchez}{Cardona et~al.}{2013}]{Cardona2013}
Cardona, G., A.~Mir, F.~Rossell{\'o}, L.~Rotger, and D.~S{\'a}nchez (2013).
\newblock Cophenetic {M}etrics for {P}hylogenetic {T}rees, {A}fter {S}okal and
  {R}ohlf.
\newblock {\em BMC Bioinformatics\/}~{\em 14\/}(1), 3.

\bibitem[\protect\citeauthoryear{Cohen, Gaubert, and Quadrat}{Cohen
  et~al.}{2004}]{COHEN2004395}
Cohen, G., S.~Gaubert, and J.-P. Quadrat (2004).
\newblock Duality and {S}eparation {T}heorems in {I}dempotent {S}emimodules.
\newblock {\em Linear Algebra and its Applications\/}~{\em 379}, 395--422.
\newblock Special Issue on the Tenth ILAS Conference (Auburn, 2002).

\bibitem[\protect\citeauthoryear{Colijn and Plazzotta}{Colijn and
  Plazzotta}{2017}]{10.1093/sysbio/syx046}
Colijn, C. and G.~Plazzotta (2017, 05).
\newblock {A {M}etric on {P}hylogenetic {T}ree {S}hapes}.
\newblock {\em Systematic Biology\/}~{\em 67\/}(1), 113--126.

\bibitem[\protect\citeauthoryear{Dress}{Dress}{1984}]{DRESS1984321}
Dress, A.~W. (1984).
\newblock Trees, tight extensions of metric spaces, and the cohomological
  dimension of certain groups: {A} note on combinatorial properties of metric
  spaces.
\newblock {\em Advances in Mathematics\/}~{\em 53\/}(3), 321--402.

\bibitem[\protect\citeauthoryear{Felsenstein}{Felsenstein}{1981}]{Felsenstein1981}
Felsenstein, J. (1981).
\newblock Evolutionary {T}rees from {D}{N}{A} {S}equences: {A} {M}aximum
  {L}ikelihood {A}pproach.
\newblock {\em Journal of Molecular Evolution\/}~{\em 17\/}(6), 368--376.

\bibitem[\protect\citeauthoryear{Fitch}{Fitch}{1971}]{max_pars}
Fitch, W.~M. (1971).
\newblock Toward {D}efining the {C}ourse of {E}volution: {M}inimum {C}hange for
  a {S}pecific {T}ree {T}opology.
\newblock {\em Systematic Biology\/}~{\em 20\/}(4), 406--416.

\bibitem[\protect\citeauthoryear{Foulds and Graham}{Foulds and
  Graham}{1982}]{FOULDS198243}
Foulds, L.~R. and R.~L. Graham (1982).
\newblock The {S}teiner {P}roblem in {P}hylogeny is {N}{P}-{C}omplete.
\newblock {\em Advances in Applied Mathematics\/}~{\em 3\/}(1), 43--49.

\bibitem[\protect\citeauthoryear{Holmes}{Holmes}{2003}]{HOLMES200317}
Holmes, S. (2003).
\newblock Statistics for {P}hylogenetic {T}rees.
\newblock {\em Theoretical {P}opulation {B}iology\/}~{\em 63\/}(1), 17--32.

\bibitem[\protect\citeauthoryear{Jardine, Jardine, and Sibson}{Jardine
  et~al.}{1967}]{JARDINE1967173}
Jardine, C., N.~Jardine, and R.~Sibson (1967).
\newblock The {S}tructure and {C}onstruction of {T}axonomic {H}ierarchies.
\newblock {\em Mathematical Biosciences\/}~{\em 1\/}(2), 173--179.

\bibitem[\protect\citeauthoryear{Juhl, Warme, Winter, and Zachariasen}{Juhl
  et~al.}{2018}]{juhl2014geosteiner}
Juhl, D., D.~M. Warme, P.~Winter, and M.~Zachariasen (2018).
\newblock The {G}eo{S}teiner software package for computing {S}teiner trees in
  the plane: {A}n updated computational study.
\newblock {\em Mathematical Programming Computation\/}~{\em 10\/}(4), 487--532.

\bibitem[\protect\citeauthoryear{Kingman}{Kingman}{2000}]{Kingman1461}
Kingman, J. F.~C. (2000).
\newblock Origins of the {C}oalescent: 1974-1982.
\newblock {\em Genetics\/}~{\em 156\/}(4), 1461--1463.

\bibitem[\protect\citeauthoryear{Knowles}{Knowles}{2009}]{Knowles2009}
Knowles, L. (2009).
\newblock Statistical {P}hylogeography.
\newblock {\em Annual Review of Ecology, Evolution, and Systematics\/}~{\em
  40}, 593--612.

\bibitem[\protect\citeauthoryear{Lee, Li, Lin, and Monod}{Lee
  et~al.}{2021}]{lee2019tropical}
Lee, W., W.~Li, B.~Lin, and A.~Monod (2021).
\newblock Tropical optimal transport and {W}asserstein distances.
\newblock {\em Information Geometry\/}, 1--41.

\bibitem[\protect\citeauthoryear{Lin, Sturmfels, Tang, and Yoshida}{Lin
  et~al.}{2017}]{doi:10.1137/16M1079841}
Lin, B., B.~Sturmfels, X.~Tang, and R.~Yoshida (2017).
\newblock Convexity in {T}ree {S}paces.
\newblock {\em SIAM Journal on Discrete Mathematics\/}~{\em 31\/}(3),
  2015--2038.

\bibitem[\protect\citeauthoryear{Lin and Yoshida}{Lin and
  Yoshida}{2018}]{lin2016tropical}
Lin, B. and R.~Yoshida (2018).
\newblock Tropical {F}ermat--{W}eber {P}oints.
\newblock {\em SIAM Journal on Discrete Mathematics\/}~{\em 32\/}(2),
  1229--1245.

\bibitem[\protect\citeauthoryear{Long and Sullivant}{Long and
  Sullivant}{2015}]{long2015identifiability}
Long, C. and S.~Sullivant (2015).
\newblock Identifiability of 3-class {J}ukes--{C}antor mixtures.
\newblock {\em Advances in Applied Mathematics\/}~{\em 64}, 89--110.

\bibitem[\protect\citeauthoryear{Maclagan and Sturmfels}{Maclagan and
  Sturmfels}{2015}]{maclagan2015introduction}
Maclagan, D. and B.~Sturmfels (2015).
\newblock {\em Introduction to {T}ropical {G}eometry ({G}raduate {S}tudies in
  {M}athematics)}.
\newblock American Mathematical Society.

\bibitem[\protect\citeauthoryear{Monod, Lin, Yoshida, and Kang}{Monod
  et~al.}{2021}]{monod2018tropical}
Monod, A., B.~Lin, R.~Yoshida, and Q.~Kang (2021).
\newblock Tropical {G}eometry of {P}hylogenetic {T}ree {S}pace: {A}
  {S}tatistical {P}erspective.
\newblock {\em arXiv:1805.12400\/}.

\bibitem[\protect\citeauthoryear{Munch and Stefanou}{Munch and
  Stefanou}{2019}]{munch2019}
Munch, E. and A.~Stefanou (2019).
\newblock The $\ell^\infty$-{C}ophenetic {M}etric for {P}hylogenetic {T}rees as
  an {I}nterleaving {D}istance.
\newblock In {\em Research in Data Science}, pp.\  109--127. Springer.

\bibitem[\protect\citeauthoryear{Owen and Provan}{Owen and
  Provan}{2011}]{Owen:2011:FAC:1916480.1916603}
Owen, M. and J.~S. Provan (2011).
\newblock A {F}ast {A}lgorithm for {C}omputing {G}eodesic {D}istances in {T}ree
  {S}pace.
\newblock {\em IEEE/ACM Trans. Comput. Biol. Bioinformatics\/}~{\em 8\/}(1),
  2--13.

\bibitem[\protect\citeauthoryear{Peng}{Peng}{2007}]{peng2007distance}
Peng, C. (2007).
\newblock Distance based methods in phylogenetic tree construction.
\newblock {\em Neural Parallel and Scientific Computations\/}~{\em 15\/}(4),
  547.

\bibitem[\protect\citeauthoryear{Rhodes and Sullivant}{Rhodes and
  Sullivant}{2012}]{rhodes2012identifiability}
Rhodes, J.~A. and S.~Sullivant (2012).
\newblock Identifiability of large phylogenetic mixture models.
\newblock {\em Bulletin of Mathematical Biology\/}~{\em 74\/}(1), 212--231.

\bibitem[\protect\citeauthoryear{Robinson and Foulds}{Robinson and
  Foulds}{1981}]{ROBINSON1981131}
Robinson, D.~F. and L.~R. Foulds (1981).
\newblock Comparison of {P}hylogenetic {T}rees.
\newblock {\em Mathematical Biosciences\/}~{\em 53\/}(1), 131--147.

\bibitem[\protect\citeauthoryear{Rosenberg}{Rosenberg}{2003}]{Rosenberg2003}
Rosenberg, N.~A. (2003).
\newblock The {S}hapes of {N}eutral {G}ene {G}enealogies in {T}wo {S}pecies:
  {P}robabilities of {M}onophyly, {P}araphyly, and {P}olyphyly in a
  {C}oalescent {M}odel.
\newblock {\em Evolution\/}~{\em 57}, 1465--1477.

\bibitem[\protect\citeauthoryear{Schr{\"o}der}{Schr{\"o}der}{1870}]{schroder1870vier}
Schr{\"o}der, E. (1870).
\newblock Vier kombinatorische {P}robleme.
\newblock {\em Z. Math. Phys\/}~{\em 15}, 361--376.

\bibitem[\protect\citeauthoryear{Speyer and Sturmfels}{Speyer and
  Sturmfels}{2004}]{Speyer2004}
Speyer, D. and B.~Sturmfels (2004).
\newblock The {T}ropical {G}rassmannian.
\newblock {\em Advances in Geometry\/}~{\em 4\/}(3).

\bibitem[\protect\citeauthoryear{Speyer and Sturmfels}{Speyer and
  Sturmfels}{2009}]{speyer2009tropical}
Speyer, D. and B.~Sturmfels (2009).
\newblock Tropical mathematics.
\newblock {\em Mathematics Magazine\/}~{\em 82\/}(3), 163--173.

\bibitem[\protect\citeauthoryear{Steel and Penny}{Steel and
  Penny}{1993}]{Steel1993}
Steel, M.~A. and D.~Penny (1993, 06).
\newblock Distributions of {T}ree {C}omparison {M}etrics: {S}ome {N}ew
  {R}esults.
\newblock {\em Systematic Biology\/}~{\em 42\/}(2), 126--141.

\bibitem[\protect\citeauthoryear{Steel and Sz{\'e}kely}{Steel and
  Sz{\'e}kely}{2006}]{steel2006variational}
Steel, M.~A. and L.~A. Sz{\'e}kely (2006, 08).
\newblock On the variational distance of two trees.
\newblock {\em Ann. Appl. Probab.\/}~{\em 16\/}(3), 1563--1575.

\bibitem[\protect\citeauthoryear{Tang, Wang, and Yoshida}{Tang
  et~al.}{2020}]{tang2020tropical}
Tang, X., H.~Wang, and R.~Yoshida (2020).
\newblock Tropical {S}upport {V}ector {M}achine and its {A}pplications to
  {P}hylogenomics.
\newblock {\em arXiv:2003.00677\/}.

\bibitem[\protect\citeauthoryear{Tavar{\'e}}{Tavar{\'e}}{1986}]{tavare1986some}
Tavar{\'e}, S. (1986).
\newblock Some probabilistic and statistical problems in the analysis of
  {D}{N}{A} sequences.
\newblock {\em Lectures on mathematics in the life sciences\/}~{\em 17\/}(2),
  57--86.

\bibitem[\protect\citeauthoryear{Tian and Kubatko}{Tian and
  Kubatko}{2014}]{Tian2014}
Tian, Y. and L.~Kubatko (2014).
\newblock Gene {T}ree {R}ooting {M}ethods {G}ive {D}istributions that {M}imic
  the {C}oalescent {P}rocess.
\newblock {\em Molecular Phylogenetics and Evolution\/}~{\em 70}, 63--69.

\bibitem[\protect\citeauthoryear{Waterman, Smith, and Beyer}{Waterman
  et~al.}{1976}]{waterman1976some}
Waterman, M.~S., T.~F. Smith, and W.~A. Beyer (1976).
\newblock Some biological sequence metrics.
\newblock {\em Advances in Mathematics\/}~{\em 20\/}(3), 367--387.

\bibitem[\protect\citeauthoryear{Yoshida, Zhang, and Zhang}{Yoshida
  et~al.}{2019}]{Yoshida2018}
Yoshida, R., L.~Zhang, and X.~Zhang (2019, Feb).
\newblock Tropical {P}rincipal {C}omponent {A}nalysis and {I}ts {A}pplication
  to {P}hylogenetics.
\newblock {\em Bulletin of Mathematical Biology\/}~{\em 81\/}(2), 568--597.

\bibitem[\protect\citeauthoryear{Yu}{Yu}{2020}]{yu2020extreme}
Yu, L. (2020).
\newblock Extreme rays of the $\ell^\infty$-nearest ultrametric tropical
  polytope.
\newblock {\em Linear Algebra and its Applications\/}~{\em 587}, 23--44.

\end{thebibliography}

%%%%%%%%%%%%%%%%%%%%%%%%%%%%%%%%%%%%%%%%%%%%%%%%%%%
%%%%%%%%%%%%%%%%%%%%%%%%%%%%%%%%%%%%%%%%%%%%%%%%%%%
%%%%%%%%%%%%%%%%%%%%%%%%%%%%%%%%%%%%%%%%%%%%%%%%%%%

\end{document}